\documentclass[nonumbers]{siamart220329}

\usepackage{amsmath,amsfonts,amssymb}
\usepackage{amsfonts,ulem}
\usepackage{graphicx}
\usepackage{epstopdf}
\usepackage{algpseudocode}
\usepackage{enumerate}
\usepackage{multicol}
\usepackage{comment}

\newcommand{\na}{\nabla}
\newcommand{\RR}{\mathbb{R}}

\newcommand{\AAA}{\mathcal{A}}

\newcommand{\KKK}{\mathcal{K}}
\newcommand{\LLL}{\mathcal{L}}
\newcommand{\III}{\mathcal{I}}

\newcommand{\MMM}{\mathcal{M}}
\newcommand{\SSS}{\mathcal{S}}

\newcommand{\lam}{\lambda}

\newcommand{\obj}{{\text{obj}}}
\newcommand{\sol}{{\text{sol}}}
\newcommand{\norm}[1]{{\|#1\|}}
\newcommand{\half}{{\tfrac{1}{2}}}

\newcommand{\yl}[1]{{\color{teal}#1}}
\newcommand{\ew}[1]{{{\color{magenta}(#1)}}}

\DeclareMathOperator*{\locmin}{locmin}

\newsiamremark{remark}{Remark}
\newsiamremark{hypothesis}{Hypothesis}
\newsiamremark{example}{Example}
\crefname{hypothesis}{Hypothesis}{Hypotheses}
\newsiamthm{claim}{Claim}
\newsiamthm{assumption}{Assumption}

\newcommand{\Footnote}[1]{}

\newcommand{\SOUT}[1]{}

\newcommand{\SOUTr}[1]{}

\headers{Decomposition for Nonconvex Two-Stage Optimization}{Y. Lou, X. Luo, A. W\"achter, and E. Wei}

\title{A Decomposition Framework for Nonlinear Nonconvex Two-Stage Optimization}

\author{
Yuchen Lou\thanks{Department of Industrial Engineering and Management Sciences, Northwestern University.  This author was partially supported by National Science Foundation grant DMS-2012410.  E-mail: \email{yuchenlou2026@u.northwestern.edu}}\and
Xinyi Luo\thanks{Department of Industrial Engineering and Management Sciences, Northwestern University.  These authors were partially supported by National Science Foundation grant DMS-2012410 and by the U.S.\ Department of Energy (DOE) Office of Electricity (OE) Advanced Grid Modeling (AGM) Research Program under program manager Ali Ghassemian.  E-mail: \email{xinyiluo2023@u.northwestern.edu}, \email{andreas.waechter@northwestern.edu}} \and
Andreas W\"achter\footnotemark[2] \and
Ermin Wei\thanks{Department of Electrical and Computer Engineering, Northwestern University.  This author was partially supported by National Science Foundation grant CMMI-2024774.  E-mail: \email{ermin.wei@northwestern.edu}}}

\begin{document}

\maketitle

\begin{abstract}
We propose a new decomposition framework for continuous nonlinear constrained two-stage optimization, where both first- and second-stage problems can be nonconvex.
A smoothing technique based on an interior-point formulation renders the optimal solution of the second-stage problem differentiable with respect to the first-stage parameters.
As a consequence, efficient off-the-shelf optimization packages can be utilized.  We show that the solution of the nonconvex second-stage problem behaves locally like a differentiable function so that existing proofs can be applied {to prove the convergence of the iterates to first-order optimal points for} the first-stage.  
We also prove fast local convergence of the algorithm as the barrier parameter is driven to zero.  
Numerical experiments for large-scale instances demonstrate the computational advantages of the decomposition framework.


\end{abstract}



\section{Introduction}
\begin{sloppypar}
In this paper, we propose an algorithmic decomposition framework that is capable of utilizing efficient second-order methods for nonlinear two-stage problems, where the first-stage (master) problem is given by 
\end{sloppypar}
\noindent
\begin{equation} \label{eq:intro_master}
\begin{split}
    \min_{x\in\mathbb{R}^{n_0}} \quad & f_0(x) + \sum_{i=1}^N\hat{f}_i(x) \\
        {\rm s.t.} \quad & c_0(x) \leq 0,
\end{split}
\end{equation}
and the second-stage problems (subproblems) are given by
\begin{equation} \label{eq:intro_subproblem}
\begin{split}
    \hat{f}_i(x) = \min_{y_i\in\mathbb{R}^{n_i}}\quad & f_i(y_i;x) \\
        {\rm s.t.} \quad & c_i(y_i;x) \leq 0.
\end{split}
\end{equation}
Here, $f_i:\RR^{n_i}\to\RR$ and $c_i:\RR^{n_i}\to\RR^{m_i}$ $(i=0,\ldots,N)$ are assumed to be sufficiently smooth but not necessarily convex.
{In~\eqref{eq:intro_subproblem}, we adopt the standard convention in nonlinear optimization of writing $\hat{f}_i(x)$ with the $\min$ operator to reflect the optimization goal.  
For nonconvex problems, however, $\hat{f}_i(x)$ is practically treated as set-valued due to the presence of multiple local minima encountered by solvers.  
We address this subtlety in more detail later after~\eqref{eq:intro_subproblem_local}.}
%

Problems of this type arise in a wide range of practical applications, such as transportation \cite{barbarosoǧlu2004two,liu2009two} and optimal power flow for electricity system \cite{lubin2011scalable,roald2023power,tu2021two,tu2020two}.
%
The individual subproblems typically correspond to different random scenarios for estimating an expected value by sample average approximation, and it is often desirable to choose a large $N$ for an accurate estimation.
%
In some circumstances, the decomposition may correspond to other criteria of partition, such as a geographic separation of electrical transmission and distribution systems \cite{tu2020two}.

Since the subproblems are defined independently of each other when $x$ is specified, decomposition algorithms are attractive due to their abilities to exploit parallel computing resources by solving subproblems simultaneously.  This becomes increasingly desirable as $N$ grows larger.
Moreover, decomposition algorithms are essential when problem instances  exceed the memory capacity of a single machine.
Distributing subproblems across multiple compute nodes effectively addresses this issue.



A major challenge for solving two-stage problems is that the value functions $\hat{f}_i(x)$ are not necessarily differentiable.
{For example}, derivatives typically do not exist at those $x$'s, where the set of constraints active at the optimal solution for~\eqref{eq:intro_subproblem} changes as $x$ varies.
{Nonsmoothness can also arise from changes in the sign of the Hessian as $x$ varies; see Example~\ref{ex:curvature}.}
%
This prevents the direct application of efficient gradient-based optimization methods\SOUT{ to solve the master problem directly}. 
While methods from nonsmooth optimization can in principle be applied, their performance is expected to be inferior to approaches that leverage the underlying smoothness of the subproblems~\eqref{eq:intro_subproblem}.

Benders' decomposition is the most commonly used approach for instances with \textit{convex} second-stage problems \cite{bnnobrs1962partitioning,geoffrion1972generalized}.
This technique maintains a master problem where the nonsmooth second-stage functions $\hat f_i$ are approximated by an increasing set of supporting hyperplanes, corresponding to subgradients of $\hat f_i$.  The hyperplanes are computed iteratively from the optimal solutions of the subproblems for different values of $x$.  This approach is also able to handle discrete variables in the first stage.

However, few methods have been developed for instances with nonconvex second-stage problems \cite{yoshio2021nested}.
In principle, one could consider computing the global minima of nonconvex subproblems, but this is in general NP-hard and requires sophisticated global optimization methods, such as spatial branch-and-bound \cite{neumaier2004complete}.
\SOUT{These methods are also limited in the size of the problems they can handle and require a lot of computation time.}
In contrast, we focus on the more practical use of Newton-based second-order methods, for which efficient and robust software implementations exist.
However, these are only guaranteed to find local minima or stationary points for the second-stage problems.
This might result in multiple candidates for values of $\hat{f}_i(x)$, meaning $\hat{f}_i(x)$ is not a well-defined function.
Strictly speaking, it is more precise to define the set-valued function
\begin{equation} \label{eq:intro_subproblem_local}
\begin{split}
    \hat{f}_i(x) \in \locmin_{y_i\in\mathbb{R}^{n_i}}\ & f_i(y_i;x) \\
        {\rm s.t.} \ & c_i(y_i;x) \leq 0,
\end{split}
\end{equation}
where $\locmin$ stands for the set of local minima of the second-stage problem.
{
Nevertheless, the main contribution of this paper is an algorithm that identifies a smooth trajectory of local solutions, so that $\hat{f}_i$ is locally well-defined.
{
Throughout the paper, our analysis focuses on the set-valued map $\hat{f}_i$ defined in~\eqref{eq:intro_subproblem_local}, though we occasionally refer to~\eqref{eq:intro_subproblem} to describe the subproblem objective, following standard practice in the nonconvex optimization literature.  
For notational simplicity, we continue to write $\hat{f}_i(x) = \min \ldots$, implicitly assuming that a particular local minimum has been selected and followed.  
This convention facilitates applying functional operations to $\hat{f}_i(x)$.
We acknowledge that such a selection is nontrivial and formalize it later via the notions of \textit{solution maps} (Section~\ref{sec:behaviors_nonconvex}) and \textit{reference points} (Section~\ref{sec:global_convergence}).  
Our analysis shows that, under mild assumptions, the proposed algorithm asymptotically identifies a local minimum, thereby resolving the set-valued ambiguity.
}
}

\subsection{Related research}

As mentioned, Benders' decomposition is a classical method for solving two-stage optimization problems when the functions involved are linear \cite{bnnobrs1962partitioning}. 
For cases where the functions are nonlinear but convex, several extensions have been proposed, such as generalized Benders' decomposition \cite{geoffrion1972generalized} and the augmented Lagrangian method \cite{rockafellar1991scenarios,ruszczynski1995convergence}. 
To address problems without the convexity, Braun introduced the framework of collaborative optimization \cite{braun1996collaborative} but it may fail to converge to a minimizer due to degeneracy \cite{alexandrov2002analytical,demiguel2000analysis}. 

In \cite{demiguel2006local}, a gradient-based method for nonlinear two-stage problems was proposed, based on $l_1$- and $l_2$-penalty smoothing of $\hat{f}_i(x)$,
but tuning the penalty parameter can be challenging \cite{brown2006evaluation}.
A sequence of recent works \cite{borges2021regularized,demiguel2008decomposition,luo2023efficient,mehrotra2007decomposition,tu2021two,tu2020two,yoshio2021nested}, including this paper, consider using a log-barrier smoothing technique. 
\cite{demiguel2008decomposition,tu2021two} provided the fundamental framework of log-barrier smoothing in two-stage optimization, and \cite{luo2023efficient} illustrated an efficient algorithm implementation.
\cite{borges2021regularized} further introduced a Tikhonov regularization term into $\hat{f}_i(x)$ and analyzed its asymptotic behaviors.
We also note that two-stage optimization can be viewed as a special case of bilevel optimization \cite{dempe2020bilevel}.
In particular, \cite{ishizuka1992double,jiang2024barrier,lin2003smoothing,shimizu1981new} share a similar methodology of smoothing and evaluation of derivatives as this paper.
However, none of these works study the case when the subproblems~\eqref{eq:intro_subproblem} are nonconvex.

Alternatively, one can reformulate the two-stage problem as an undecomposed single-stage problem, where $\hat f_i$ are substituted into the master problem to obtain one large monolithic optimization problem; see~\eqref{eq:undecomp_Obj} in Section~\ref{sec:global_convergence}.
This monolithic problem can then be solved using an interior-point method, leveraging parallelizable decomposition techniques for the associated linear systems, e.g., the Schur complement method \cite{demiguel2008decomposition,lubin2011scalable,zavala2008interior}. 
{A notable advantage of our approach is its ease of initialization by a presolve: we begin by solving the first-stage problem once with the second-stage variables fixed. 
Additionally, this method offers the potential benefit of decomposing highly nonlinear instances into subproblems that may exhibit faster convergence.}

For fast local convergence, \cite{demiguel2006local} proved a superlinear rate for a decomposition algorithm whose smoothing parameter is fixed.
\cite{demiguel2008decomposition,tammer1987application} established the superlinear rate for decreasing smoothing parameters under, however, a rather restrictive assumption: the linear independence constraint qualification (LICQ) holds for all the subproblems \cref{eq:intro_subproblem}, also known as Strong LICQ (SLICQ) \cite{tammer1987application}.
SLICQ is not likely to hold in practice (see Example~\ref{ex:bilinear}), and we do not assume it in our analysis.

\subsection{Contributions and outline}

\Footnote{\ew{I'm ok with the format here, would personally prefer a more bullet point style with emphasis on a few points}}
Our work goes beyond previously proposed methods for nonlinear two-stage optimization.
%
To the best of our knowledge, it is the first method that handles nonconvexity of the subproblems in a natural manner.
Our method is capable to seek local solutions and utilizes state-of-the-art nonlinear optimization algorithms and their efficient software implementations. \SOUT{for both the first- and second-stage problems.}

After introducing the proposed smoothing technique in Section~\ref{sec:obj_smoothing}, we demonstrate in Section~\ref{sec:sol_smoothing} that a smoothing approach used in previous works \cite{borges2021regularized,demiguel2006local,demiguel2008decomposition} has the undesired property that it can introduce nonconvexity {and lead to spurious solutions}, even for convex second-stage instances. 
\SOUT{As a result, it may lead to spurious solutions that do not correspond to stationary points of the optimization problem. }This does not occur for the smoothing technique used in this work.

In Section~\ref{sec:behaviors_nonconvex}, we explore the challenges caused by the non-uniqueness of local minima in nonconvex second-stage problems.   
By providing small concrete examples, we give the intuition behind our proposed concept of solution maps that make it possible to define a local second-stage value function $\hat f_i$.  
After stating the decomposition framework formally in Section~\ref{sec:algorithm}, we prove in Section~\ref{sec:hat_f_is_C2} that a warm-starting mechanism for a second-order subproblem solver computes such a function locally.  
In Section~\ref{sec:SQP_global_convergence_fixed_mu} we show that this result enables us to extend existing \SOUTr{global }convergence proofs from nonlinear optimization to the master problem.  
As a specific example, we consider a sequential quadratic programming (SQP) method with an $\ell_1$-penalty function.
In Section~\ref{sec:global_conv_changing_mu} we prove the asymptotic \SOUTr{global }convergence with diminishing smoothing parameters.

We also propose, in Section~\ref{sec:local_convergence}, a strategy that yields a provable superlinear convergence rate of the overall algorithm under standard nondegeneracy assumptions.\SOUT{In the limit, it requires only one iteration of the second-stage solver.}  
Importantly, we demonstrate that the strategy can still be executed in a distributed manner, with computations readily available from the original framework.

Finally, in Section~\ref{sec:numerical}, we examine the practical performance of the proposed framework. Our C++ implementation is based on an SQP solver and an interior-point solver.
It is validated that our framework outperforms a state-of-the-art nonlinear optimization solver and can benefit well from parallel computational resources.

\subsection{Notation}
Throughout the paper, $|\cdot|$ denotes the $\ell_1$-norm and $\|\cdot\|$ denotes the $\ell_2$-norm.
Unless specified, the vector spaces considered in this paper are coped with the $\ell_2$-norm.
Given a vector $x$, we write the vector space in which $x$ stays as $\mathbb{R}^{n_x}$ if the dimensionality is not specified beforehand.
Given $x\in\mathbb{R}^{n_x}$ and $r>0$, we write $B(x,r)$ as the open ball centered at $x$ with radius $r$, {and $\bar B(x,r)$ as the closure of $B(x,r)$.}
\SOUTr{Given $(x,y)\in\mathbb{R}^{n_x+n_y}$ and $r>0$, $B_x((x,y),r)$ denotes the projection of $B((x,y),r)\subset \mathbb{R}^{n_x+n_y}$ onto $\mathbb{R}^{n_x}$.}

\section{Smoothing the second-stage problem}
\label{sec:smoothing}

In this section, we introduce two differentiable approximations $\hat f_i(x;\mu)$ of $f_i(x)$ depending on a smoothing parameter $\mu > 0$, such that $\lim_{\mu \to 0} \hat{f}_i(x;\mu) = \hat{f}_i(x)$ for all $x$.
\SOUT{We will also discuss how to compute the derivatives of $\hat{f}_i(\cdot;\mu)$, so that smooth optimization methods can be implemented.}
Correspondingly, we define a smoothed master problem as
\begin{equation} \label{eq:smooth_master}
\begin{split}
    \min_{x} \quad & f_0(x) + \sum_{i=1}^N \hat{f}_i(x;\mu) \\
        {\rm s.t.} \quad & c_0(x)\leq 0. 
\end{split}
\end{equation}
The basic idea of the proposed algorithm is to solve \eqref{eq:smooth_master} repeatedly for diminishing values of $\mu$.
{Because $\hat f_i(\cdot; \mu)$ is constructed to be differentiable, it is amenable to standard gradient-based methods.  
Moreover, if  $f_i$ and $c_i$ are sufficiently smooth, $\hat f_i(\cdot; \mu)$ will be twice differentiable, allowing the use of second-order methods as well.
This smooth reformulation is key to our goal of enabling fast practical convergence by leveraging second-order optimization techniques.}

\subsection{Objective smoothing}
\label{sec:obj_smoothing}
The nonsmoothness of the value functions  $\hat{f}_i(x)$ can be caused by a change of the set of active constraints  at the optimal solution as $x$ is varied.  
To address this issue, we convert the nonlinear inequality constraints into equality constraints by introducing nonnegative slack variables.
Then, they are handled by $\log$-barrier terms that are added to the objective function.
This leads us to the well-known barrier-function formulation of the subproblem:
\begin{subequations} \label{eq:obj_smoothing}
\begin{align}
    \hat{f}^\obj_i(x;\mu) := \min_{y_i,s_i}\quad & f_i(y_i;x)-\mu \sum_j \ln(s_{ij}) \\
        {\rm s.t.} \quad & c_i(y_i;x)+s_i = 0. && [\lam_i]\label{eq:obj_smoothing_constr}
\end{align}
\end{subequations}
We name this smooth approximation $\hat{f}^\obj_i(x;\mu)$ as \textit{objective smoothing}.
Here, $\mu>0$ is the barrier parameter, and it is well-known that solutions of the original subproblem \eqref{eq:intro_subproblem} can be recovered as limit points of optimal solutions of the barrier problem \eqref{eq:obj_smoothing} as $\mu\to0$ \cite{nocedal1999numerical}.
In our context, we can interpret $\mu$ as a parameter that determines the degree of smoothing.
The vector $\lam_i$ denotes the multipliers for the constraints \eqref{eq:obj_smoothing_constr}.

In order to enable an easier evaluation of derivatives of $\hat f_i^{\rm obj}(\cdot,\mu)$, we let $\tilde x_i\in\mathbb{R}^{n_0}$ be a copy of $x$ in the $i$-th subproblem, and equivalently rewrite~\eqref{eq:obj_smoothing} as
\begin{subequations} \label{eq:obj_smoothing_proj}
\begin{align}
    \hat{f}^\obj_i(x;\mu) := \min_{y_i,s_i,\tilde x_i}\quad & f_i(y_i;\tilde x_i)-\mu \sum_j \ln(s_{ij}) \\
        {\rm s.t.} \quad & c_i(y_i;\tilde x_i)+s_i = 0, && [\lam_i] \\
        & \tilde x_i-x = 0. && [\eta_i]\label{eq:x_copy_constr}
\end{align}
\end{subequations}
%

In order to compute derivatives of $\hat f^\obj_i(x;\mu)$, we first introduce the primal-dual first-order KKT optimality conditions for \eqref{eq:obj_smoothing_proj}, namely
\begin{equation}\label{eq:smooth_subprob_optcond}
  F_i(y_i, \tilde x_i, s_i, \lambda_i,\eta_i; x, \mu) =
  \begin{pmatrix}
   \na_{y_i} \LLL_i (y_i, \tilde x_i, s_i,\lambda_i,\eta_i; x)\\
   \na_{\tilde x_i} \LLL_i (y_i, \tilde x_i, s_i,\lambda_i,\eta_i; x) \\
    s_i \circ \lambda_i - \mu e \\
        c_i(y_i;\tilde x_i) + s_i & \\
        \tilde x_i-x 
  \end{pmatrix}
        = 0,
\end{equation}
where $\circ$ stands for the element-wise product, $e$ is the vector of all ones in $\mathbb{R}^{m_i}$, and $\mathcal{L}_i$ is the Lagrangian function corresponding to the smoothed subproblem \eqref{eq:obj_smoothing_proj}, i.e., 
\begin{equation*}
    \mathcal{L}_i(y_i, \tilde x_i, s_i, \lambda_i,\eta_i; x) = f_i(y_i;\tilde x_i) + (c_{i}(y_i;\tilde x_i)+s_i)^T \lambda_i+(\tilde x_i-x)^T\eta_i.
\end{equation*}
Let 
\begin{equation}
\label{eq:v_i_definition}
    v_i^*(x;\mu) = (y_i^*(x;\mu), \tilde x_i^*(x;\mu), s_i^*(x;\mu), \lambda_i^*(x;\mu), \eta_i^*(x;\mu))
\end{equation}
denote a KKT point for a given $x$, i.e., $F_i(v_i^*(x;\mu);x,\mu)=0$.
Note $\tilde x_i^*(x;\mu)=x$, and both $s_i^*$ and $\lambda_i^*$ are positive.
Assuming that $\na_{v_i} F_i(v_i^*(x;\mu);x,\mu)$ is nonsingular (see Section~\ref{sec:nondegenerate_justification} for justification), by sensitivity analysis \cite[ Chapter 11.7]{luenberger1984linear} and the implicit function theorem 
we have that
\begin{equation}\label{eq:na_hat_f}
    \na_x \hat{f}_i^{\obj}(x;\mu)=-\eta_i^*(x;\mu),\quad \na^2_{xx}\hat{f}_i^{\obj}(x;\mu)=-\na_x\eta_i^*(x;\mu).
\end{equation}
Note that the optimal multipliers $\eta_i^*$ are usually an output of the subproblem solver and can be obtained without extra work.
Furthermore, the implicit function theorem yields that $\na_{x}v_i^*(x;\mu)$ can be computed as the solution of the (matrix) linear system
\begin{equation}
\label{eq:na_vstar}
\na_{v_i}F_i(v^*_i(x;\mu);x,\mu)^T\na_{x}v_i^*(x;\mu)^T = - \na_{x}F_i(v^*_i(x;\mu);x,\mu)^T,
\end{equation}
which gives the second derivatives of $\hat f_i^{\obj}$ according to~\eqref{eq:na_hat_f}.

From a computational perspective it is beneficial to notice that, in {many} applications, a subproblem~\eqref{eq:intro_subproblem} depends only on a small subvector $x_i$ of $x$; {see, e.g.,~\cite{lubin2011scalable,tu2020two,yoshio2021nested}, where the modeling typically introduces master variables into the subproblems via $P_ix$, using a projection matrix $P_i$ that selects the relevant subvector}.
As a consequence, the right-hand side of \eqref{eq:na_vstar} has only as many columns as $n_{x_i}$ instead of $n_x$, resulting in much less work.
Also, the left-most matrix in the linear system \eqref{eq:na_vstar} is identical to the one that a Newton-based algorithm for \eqref{eq:obj_smoothing_proj} uses in every iteration; see, e.g., Algorithm~\ref{alg:newton}.  Therefore, the internal linear algebra routines in such an algorithm can be utilized for \eqref{eq:na_vstar} without much additional programming effort.


\subsection{Solution smoothing}
\label{sec:sol_smoothing}

As an alternative to the previous approach, we define the approximation called \textit{solution smoothing} as
\begin{equation}
\label{eq:sol_smoothing}
    \hat f_i^\sol(x;\mu)=f_i(y_i^*(x;\mu);x)
\end{equation}
with the subvector $y_i^*(x;\mu)$ of $v_i^*(x;\mu)$ solving~\eqref{eq:smooth_subprob_optcond}.  
Since $y_i^*(x;\mu)$ is differentiable by the implicit function theorem, the chain rule implies that $\hat f_i^\sol(\cdot;\mu)$ is also differentiable. 

However, applying the chain rule to \eqref{eq:sol_smoothing} twice to get $\na^2_{xx}\hat f_i^\sol(x;\mu)$ results in the necessity of computing the second derivatives of $v_i^*(x;\mu)$, which requires more work than solving \eqref{eq:na_vstar}.  Furthermore, this procedure involves second derivatives of $F_i$, which requires computing the third derivatives of $f_i$ and $c_i$; see \eqref{eq:smooth_subprob_optcond}.

This approach has been used in \cite{borges2021regularized,tu2020two}.
In addition to the increased computational costs compared to objective smoothing, it has another significant drawback.
Suppose the original subproblem~\eqref{eq:intro_subproblem} is convex with respect to $x$ and $y_i$, then classical convex analysis guarantees that $\hat{f}^\obj_i(x;\mu)$~\eqref{eq:obj_smoothing} is convex;
see, e.g., \cite[Lemma 3]{borges2021regularized}.\SOUT{\cite{still2018lectures} [Theorem 5.2] and}\SOUT{This is also evident in Figure~\ref{fig:smoothing_example}.}
However, $\hat f_i^\sol(x;\mu)$ can be nonconvex.
To see this, let us consider the following example.

\begin{example}\label{ex:sol_smoothing}
\begin{equation} \label{eq:nonconvex_sample}
\begin{aligned}
    \min_{x\in\mathbb{R}}\quad & \hat{f}_1(x)  &\qquad \hat{f}_1(x)=\min_{y_{11},y_{12}\in\mathbb{R}}\quad & \tfrac{3}{2}\sqrt{2}y_{11}-\tfrac{1}{2}\sqrt{2}y_{12} \\
        {\rm s.t.} \quad &  x\in[0.1,2] & {\rm s.t.} \quad & y_{11}+y_{12}=x, \quad y_{11},y_{12}\geq 0.
\end{aligned}
\end{equation}
\end{example}
%
%
The smoothed subproblem is given by:
\begin{equation}\label{eq:nonconvex_sample_sol}
\begin{split}
    y_1^*(x;\mu):=\arg\min_{y_{11},y_{12}\in\mathbb{R}}\quad & \tfrac{3}{2}\sqrt{2}y_{11}-\tfrac{1}{2}\sqrt{2}y_{12}-\mu\sum_{j=1}^2\log(s_{1j}) \\
        {\rm s.t.} \quad & y_{11}+y_{12}=x,
        \quad -y_{11}+s_{11}=0, \quad -y_{12}+s_{12}=0.
\end{split}
\end{equation}
Note that the subproblem in \eqref{eq:nonconvex_sample} is a linear program, which is clearly convex with respect to $x$ and $y_i$ jointly.
However, $\hat f_i^\sol(x;\mu)$ is nonconvex as can be seen in Figure~\ref{fig:smoothing_example}.
Indeed, we can write out the closed form of the solutions to~\eqref{eq:nonconvex_sample_sol} as:
\begin{equation*}
    y_{11}^*(x;\mu)=\frac{\mu+\sqrt{2}x-\sqrt{\mu^2+2x^2}}{2\sqrt{2}},\quad y_{12}^*(x;\mu)=\frac{-\mu+\sqrt{2}x+\sqrt{\mu^2+2x^2}}{2\sqrt{2}}.
\end{equation*}
Substituting this into the objective yields $\hat f_i^\sol(x;\mu)=\mu+\frac{\sqrt{2}}{2}x-\sqrt{\mu^2+2x^2}$ and $\na_{xx}\hat f_i^\sol(x;\mu)=-\frac{2\mu^2}{(2x^2+\mu^2)^{3/2}}<0$.
Therefore, $\hat f_i^\sol(x;\mu)$ is nonconvex for all $\mu>0$.
In contrast, $\hat f_i^\obj(x;\mu)$ is convex; see Figure~\ref{fig:smoothing_example}.


\begin{figure}
\centering
\includegraphics[width = 0.43\textwidth]{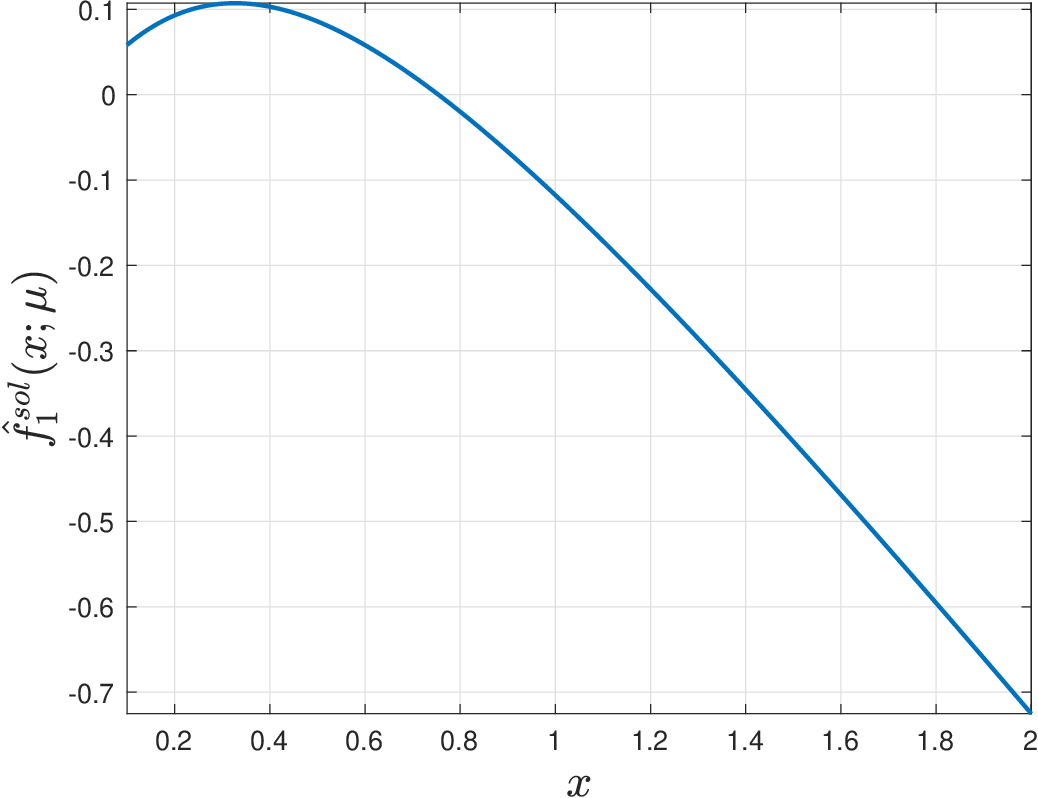}
\includegraphics[width = 0.43\textwidth]{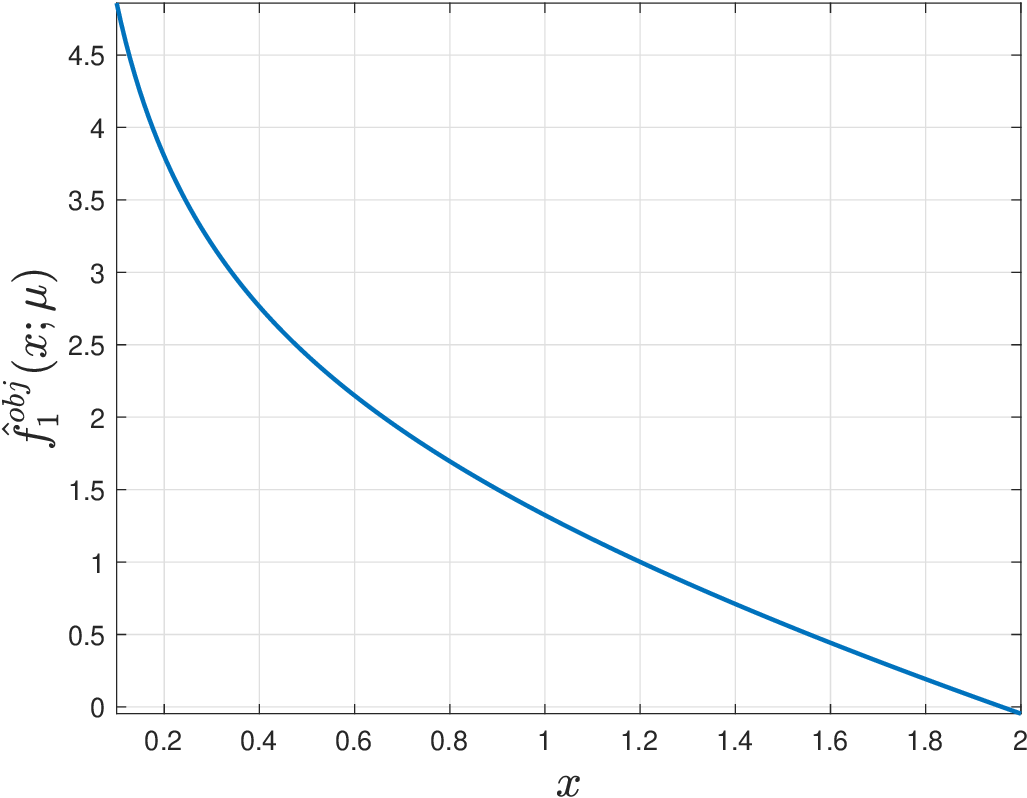}
\label{fig:smoothing_example}
\caption{Plot of $\hat{f}_1^{\sol}(x;\mu)$ (left) and plot of $\hat{f}_1^{\obj}(x;\mu)$ (right), $x\in[0.1,2]$, $\mu=1$.}
\end{figure}

In conclusion, in the context of two-stage optimization it is not preferred to use solution smoothing, compared to objective smoothing.
Nevertheless, most of the \SOUTr{global }convergence theory presented here applies to both approaches.

\section{Challenges of nonconvex subproblems}
\label{sec:behaviors_nonconvex}

The case where the subproblems~\eqref{eq:intro_subproblem} are convex is well-studied in the literature \cite{geoffrion1972generalized,rockafellar1991scenarios,ruszczynski1995convergence}. 
Convergence becomes more challenging  when the subproblems are nonconvex and only local minima of the subproblems are available.
These challenges majorly root from $\hat{f}_i(\cdot)$ being potentially a nonsmooth set-valued function; {see~\eqref{eq:intro_subproblem_local}}.
In this section, we illuminate the discussion by two specific examples showing those pathological structures. 

\subsection{Original formulation without smoothing}
We start by considering the original formulation~\eqref{eq:intro_master} and~\eqref{eq:intro_subproblem} without the smoothing introduced in Section~\ref{sec:smoothing}.

\SOUT{Let us consider the following example.}
\begin{example}\label{ex:bilinear}
\begin{equation} \label{eq:linear_example}
\begin{aligned}
    \min_{x\in\mathbb{R}} \quad & \hat{f}_1(x) & \qquad \hat{f}_1(x)=\min_{y\in\mathbb{R}} \quad & y\\
        {\rm s.t.} \quad & 0\leq x \leq 2 & {\rm s.t.} \quad & (y+1+2x)(y+x)\geq 0, \quad
         y\geq -2-x.
\end{aligned}
\end{equation}
\end{example}
%
By algebra, the feasible region of the subproblem is
\begin{equation}
\begin{cases}
    [-2-x,-1-2x]\cup [-x,\infty], & \text{ if }0\leq x\leq 1, \\
    [-x,\infty], & \text{ if }1<x\leq 2.
\end{cases}
\end{equation}
On the left panel of Figure~\ref{fig:feasible_linear}, we plot the feasible region for different values of $x$.
It can be readily seen that the subproblem is nonconvex for $0\leq x\leq 1$, since the feasible region consists of two disjoint intervals until the one on the left turns into a singleton at $x=1$.
Therefore, the subproblem has two local minimizers which are the left-endpoints of each interval, namely $-2-x$ and $-x$.
We then denote the two trajectories of local minimizers as functions of $x$:
\begin{equation}\label{eq:solpath_ex2}
    y_1^*(x)=-2-x,\quad y_2^*(x)=-x.
\end{equation}
When $1<x\leq 2$, one of the feasible interval vanishes and the feasible region is a connected interval.
As a result, $y_2^*(x)$ is the only local minimizer when $x\in(1,2]$.
On the right panel of Figure~\ref{fig:feasible_linear}, we plot both $y_1^*(x)$ and $y_2^*(x)$ which are the same as the local evaluation of $\hat{f}_1(x)$, and it can be readily seen that global optimal solution is $(x^*,y^*)=(1,-3)$.

\begin{figure}[htpt!]
    \centering
    \includegraphics[width = 0.48\textwidth]{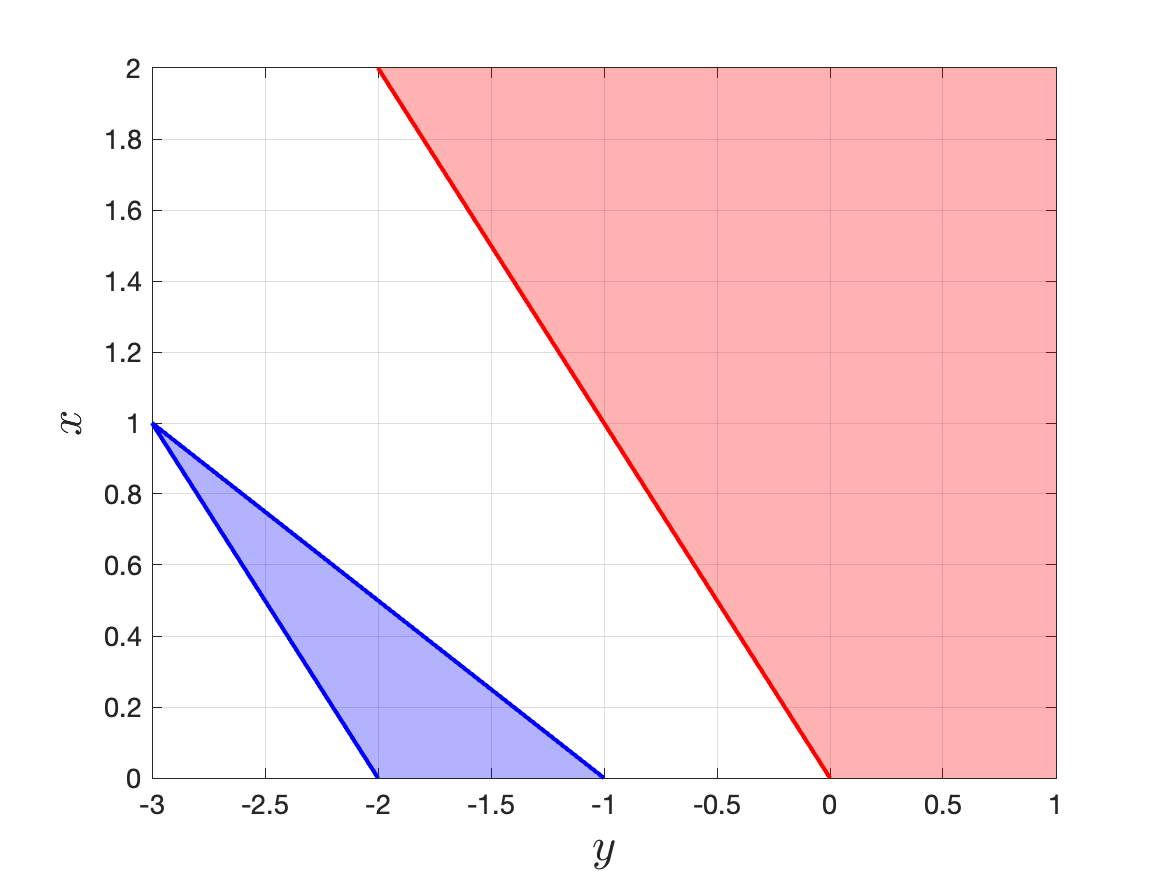}
    \includegraphics[width = 0.43\textwidth]{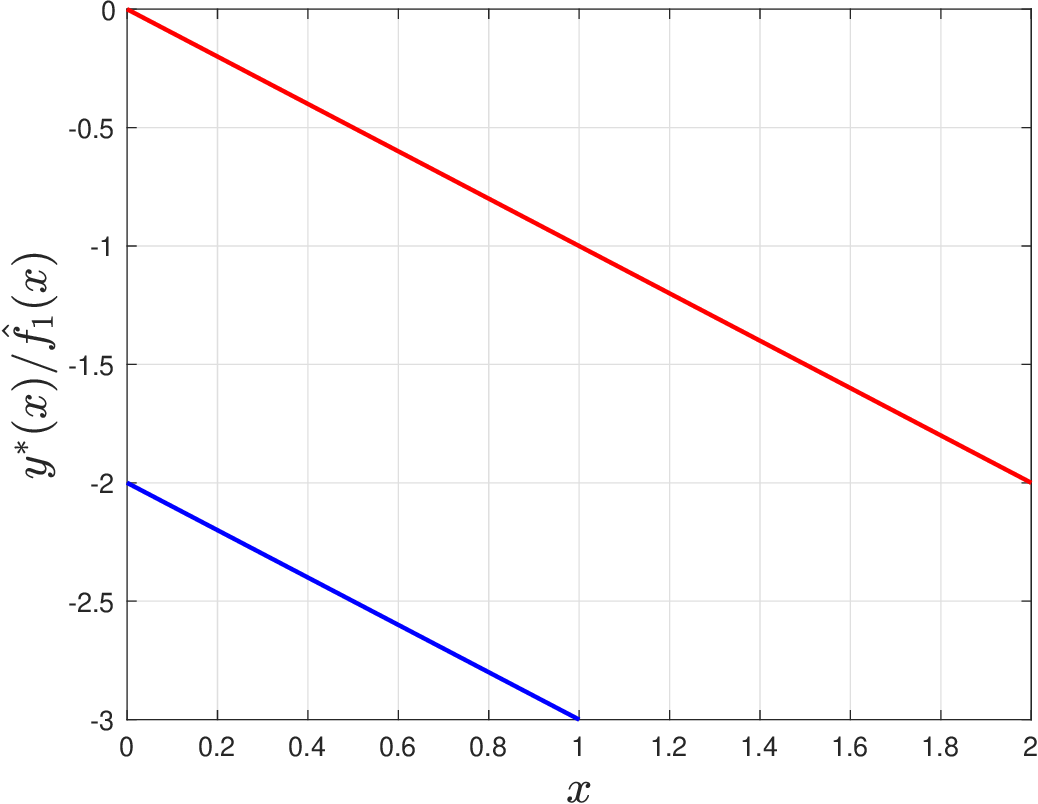}
    \caption{Left: Feasible region of the subproblem in Example~\ref{ex:bilinear} for different values of $x$.
    Right: Solution maps and $\hat{f}_1(x)$ Example~\ref{ex:bilinear}.}
    \label{fig:feasible_linear}
\end{figure}

In this paper, we call a parametric solution of the subproblem,  such as $y_1^*(x)$ and $y_2^*(x)$ above, a 
\textit{solution map}. 
To be concrete, we first define the set
\begin{equation}
    \label{aeq:def_Mi}
    \MMM_i = \{(x,y_i)\in\RR^{n_0}\times\RR^{n_i} : y_i \text{ is a local minimizer of \eqref{eq:intro_subproblem} for } x \}.
\end{equation}
Then we call $y^*_i(\cdot):U\to \mathbb{R}^{n_i}$ a solution map for the $i$-th subproblem if it is a continuous mapping from a neighborhood $U\subset\mathbb{R}^{n_0}$ with $(x,y^*_i(x))\in \MMM_i$ for all $x\in U$.

Figure~\ref{fig:feasible_linear} shows that multiple solution maps $y^*_i(x)$ may exist at a single $x$.
This indicates that $\hat{f}_i(x)$ may not be a well-defined function, but a set-valued mapping.
In addition, a  solution map might be defined only for  a subset of the feasible region; for example, $y_1^*(x)$ in \eqref{eq:solpath_ex2} vanishes at $x=1$.

If $\hat{f}_i(x)$ is computed as a local minimum of the subproblem, whenever a subproblem solver is called, its evaluation may correspond to different solution maps.
{
The uncontrollable switching among solution maps results in a discontinuous appearance of $\hat f_i(x)$, which can cause convergence issues for the master problem solver.}
%
In Section~\ref{sec:global_conv_fixed_mu} we show that a warm start strategy in the subproblem solver can overcome this challenge.


We note that similar concepts of solution maps have been explored in the contexts of parametric optimization \cite{guddat1990parametric} and time-varying optimization \cite{fattahi2020absence}, although merely in a single-parametric fashion i.e., $x\in\mathbb{R}$.
\cite{schecter1986structure} studied\SOUT{ the topological structure of} parametric optimization in a multi-dimensional setting, but without any algorithmic design.

\SOUT{The previous example concerns the changes of feasible region as the master variable $x$ changes.}
Next, we introduce another example where the curvature of the subproblem changes the sign when $x$ is varied.
\begin{example}\label{ex:curvature}
    \begin{equation} \label{eq:QP_example}
\begin{aligned}
    \min_{x\in\mathbb{R}} \quad & \hat{f}_1(x)  & \quad \hat{f}_1(x)=\min_{y\in\mathbb{R}} \quad & xy^2 \\
        {\rm s.t.} \quad & -1\leq x \leq 1 & {\rm s.t.} \quad & -1\leq y \leq 2.
\end{aligned}
\end{equation}
\end{example}


An important feature of this example is that the subproblem can be either convex or concave depending on the sign of $x$.
When $0<x\leq 1$, the subproblem is strictly convex with a global minimizer $y=0$.
When $-1\leq x<0$, it is concave with two local minimizers on the boundary: $y=-1$ and $y=2$.
If $x=0$, then any point in the feasible region $[-1,2]$ is globally optimal.
The solution maps are plotted on the left panel of Figure~\ref{fig:QP_solution_map_unsmooth}, and $\hat{f}_1(x)$ is on the right.
In addition to the local minimizers in red and blue, we also plot a stationary point (in this example, the global maximizer) in green.
We observe that as $x$ decreases, the solution maps extend to multiple branches when the convexity switches to concavity at $x=0$.
This is called ``bifurcation'' in the context of time-varying optimization \cite{fattahi2020absence}.
Such a structure can again prevent convergence, as it may confuse an algorithm about which solution map to follow.


\begin{figure}[htpt!]
    \centering
    \includegraphics[width = 0.43\textwidth]{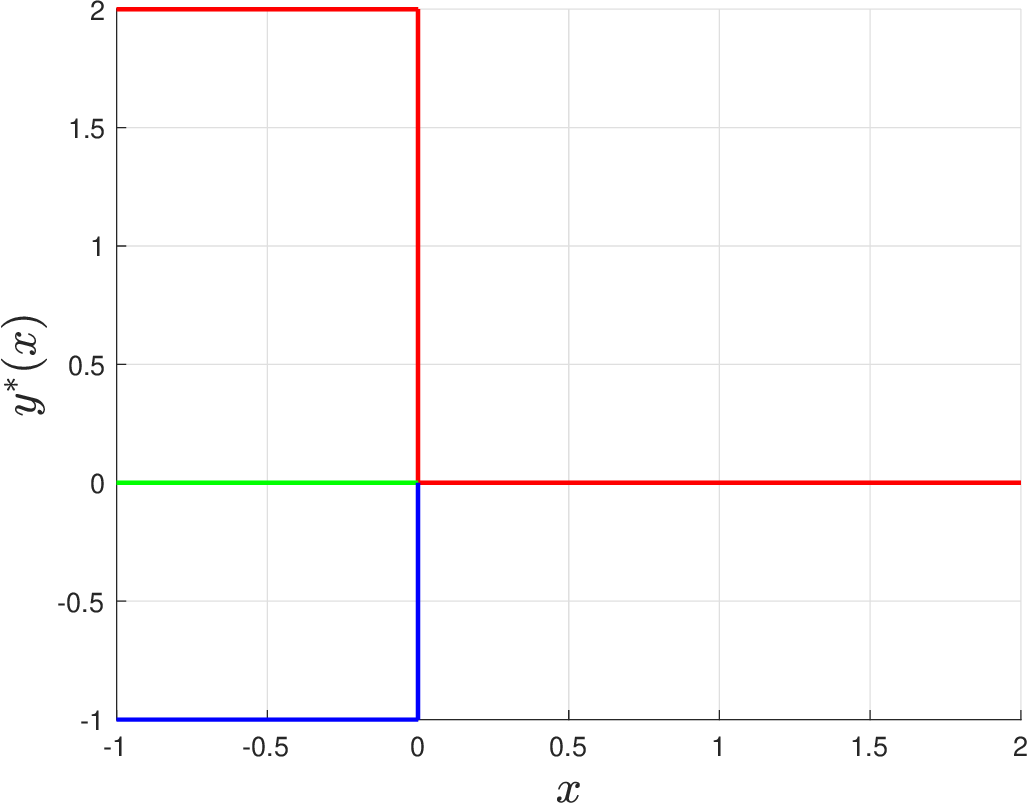}
    \includegraphics[width = 0.43\textwidth]{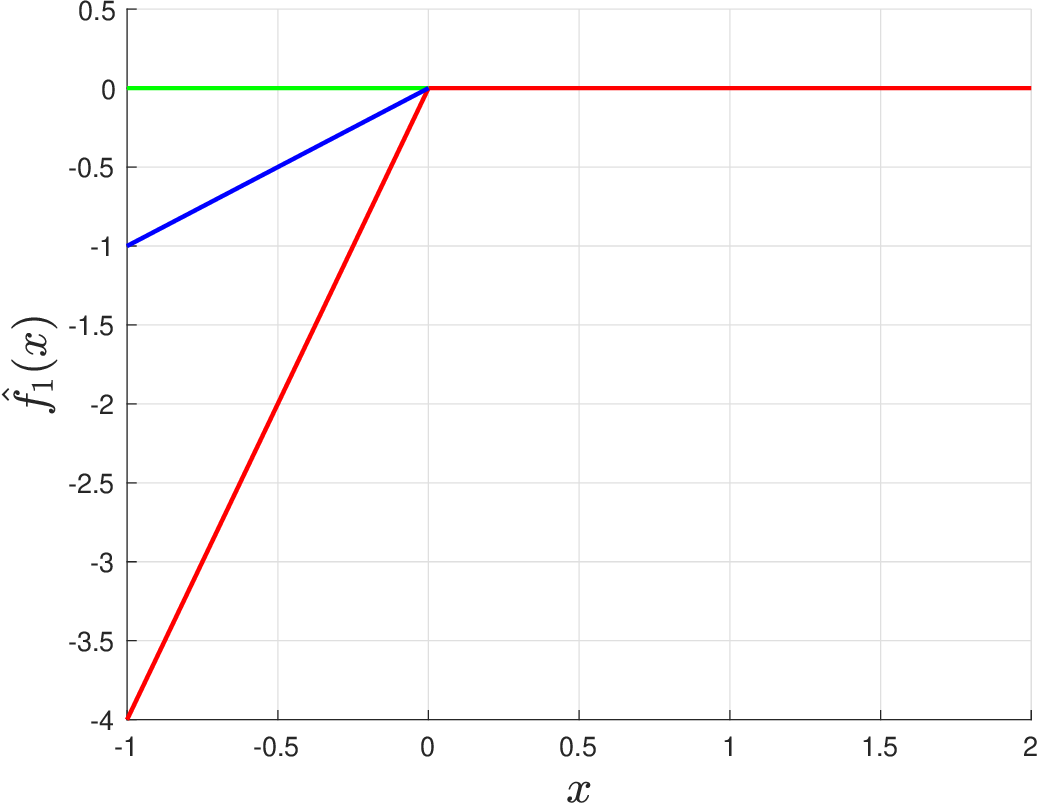}
    \caption{Plot of solution maps (left) and $\hat{f}_1(x)$ (right) for Example~\ref{ex:curvature}.}
    \label{fig:QP_solution_map_unsmooth}
\end{figure}

\subsection{Nonconvex subproblem with smoothing}
Our algorithm implements the smoothing technique detailed in Section~\ref{sec:smoothing}, and it is natural to ask if the issues described in the previous section persist when the $\log$-barrier is introduced. 

\paragraph{Example~\ref{ex:bilinear}}
With the log-barrier smoothing, the subproblem becomes
\begin{equation}
\begin{split}
\hat{f}_1(x,\mu)=\min_{y\in\mathbb{R}} \quad & y-\mu\sum_{j=1}^2\log(s_j) \\
    {\rm s.t.} \quad & (y+1+2x)(y+x)=s_1, \quad y+2+x=s_2.
\end{split}
\end{equation}
We plot the smoothed solution maps and master objective function for $\mu=1,0.5,0.1$ in Figure~\ref{fig:linear_solution_maps_smooth}.
An important observation is that the $\log$-barrier increasingly penalizes the subproblem objective function as $x$ approaches 1 from below, since the feasible interval $[-2-x,-1-2x]$ becomes increasingly narrow, and $s_1$ is forced to become arbitrarily small.
As $x\to 1$, the slack variable $s_1$ disappears, and\SOUT{ the blue solution map is open} at $x=1$ one can observe the blowing-up of $\hat{f}_1(x,\mu)$ on the right panel.

\begin{figure}[htpt!]
    \centering
    \includegraphics[width = 0.43\textwidth]{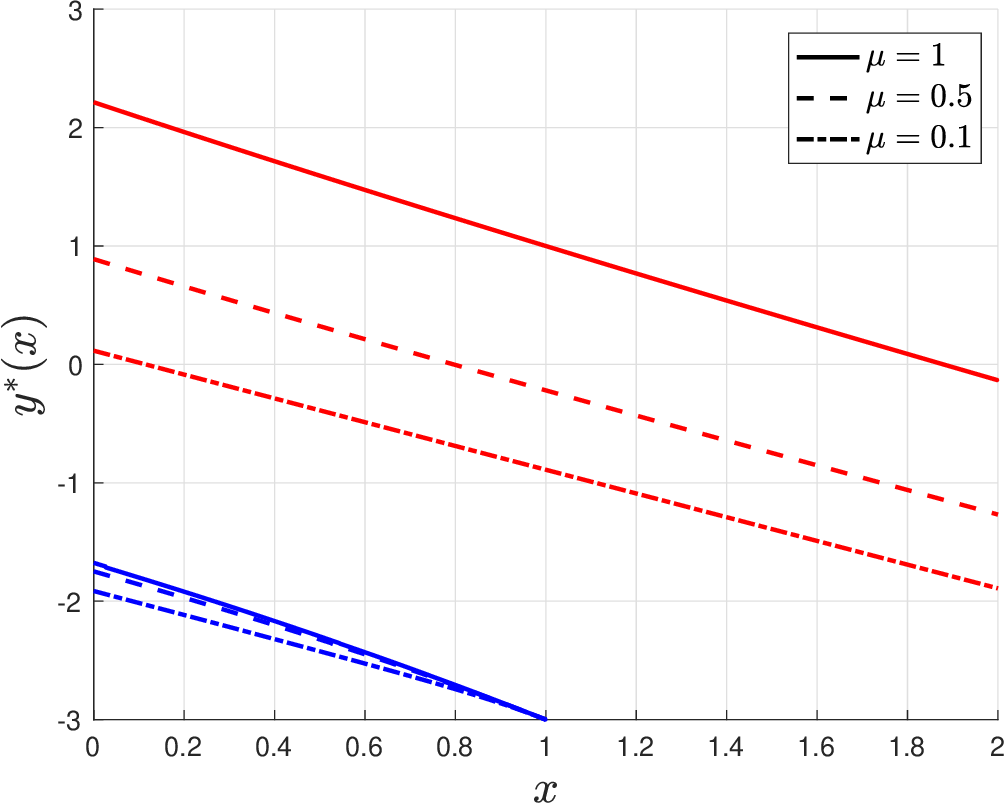}
    \includegraphics[width = 0.43\textwidth]{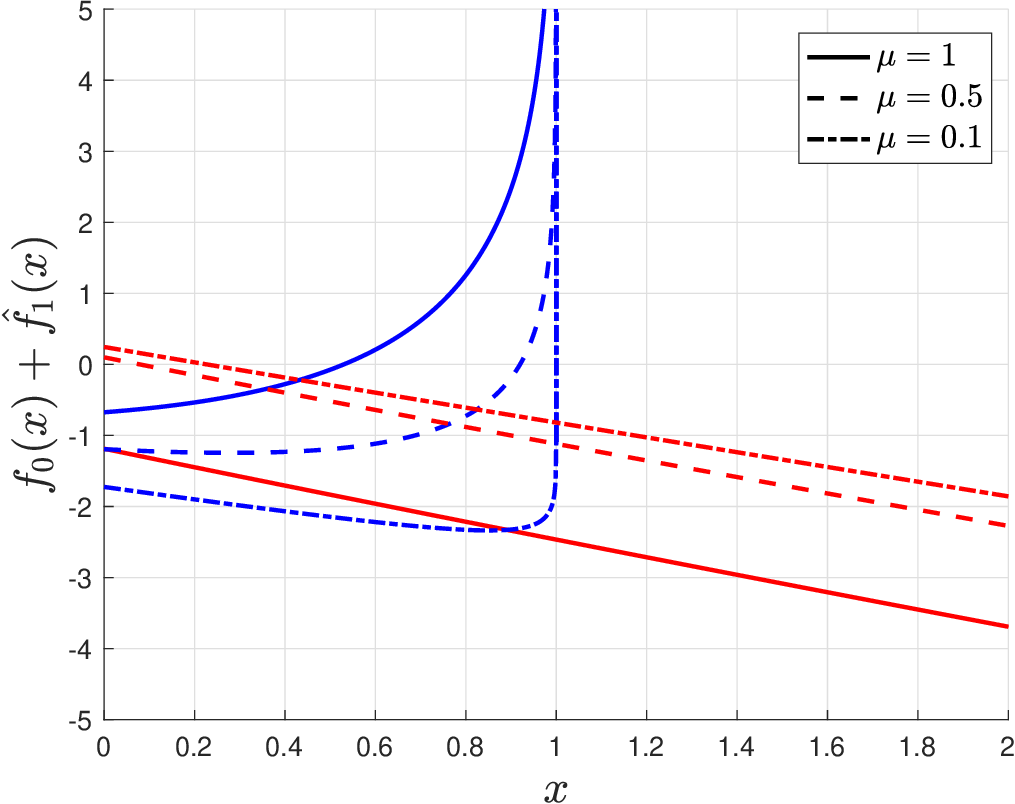}
    \caption{Solution maps (left) and smoothed master problem objective (left) for Example~\ref{ex:bilinear}.}
    \label{fig:linear_solution_maps_smooth}
\end{figure}

One advantage of smoothing is that although multiple solution maps still exist, each of them generates a smooth $\hat{f}_1(x,\mu)$.
Also note that in this example, the point $x=1$, at which the blue feasible region (see the left panel of Figure~\ref{fig:feasible_linear}) disappears, is never approached by the master solver.
This is because $\lim_{x\to1^-}\hat{f}_1(x,\mu)=+\infty$.
As $\mu\to0$, the solver can still recover the true solution; see the right panel of Figure~\ref{fig:linear_solution_maps_smooth}.

\paragraph{Example~\ref{ex:curvature}}
After smoothing, the subproblem of Example~\ref{ex:curvature} becomes
\begin{equation}
\begin{split}
    \hat{f}_1(x,\mu)=\min_{x\in\mathbb{R}} \quad & xy^2-\mu\sum_{j=1}^2\log(s_j) \\
    {\rm s.t.}  \quad & -1+s_1=y=2-s_2.
\end{split}
\end{equation}
Figure~\ref{fig:QP_solution_maps_smooth} plots solution maps and $\hat{f}_1(x,\mu)$ with $\mu=0.5,0.05,0.005$.
It can be seen that the plots recover the pattern of Figure~\ref{fig:QP_solution_map_unsmooth} as $\mu\to0$.
The curves in red are smooth for $\mu>0$, while the blue and green trajectories still intersect, {which represents a type of inefficiency to be addressed in Section~\ref{sec:nondegenerate_justification}.}

\begin{figure}[htpt!]
    \centering
    \includegraphics[width = 0.43\textwidth]{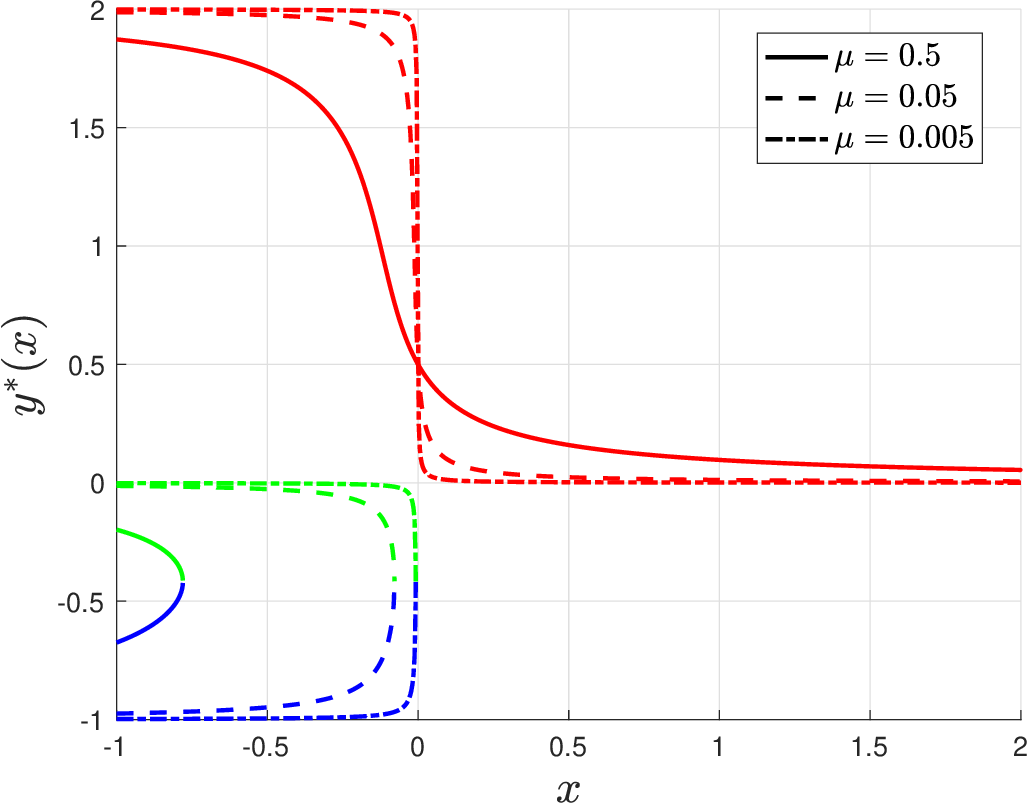}
    \includegraphics[width = 0.43\textwidth]{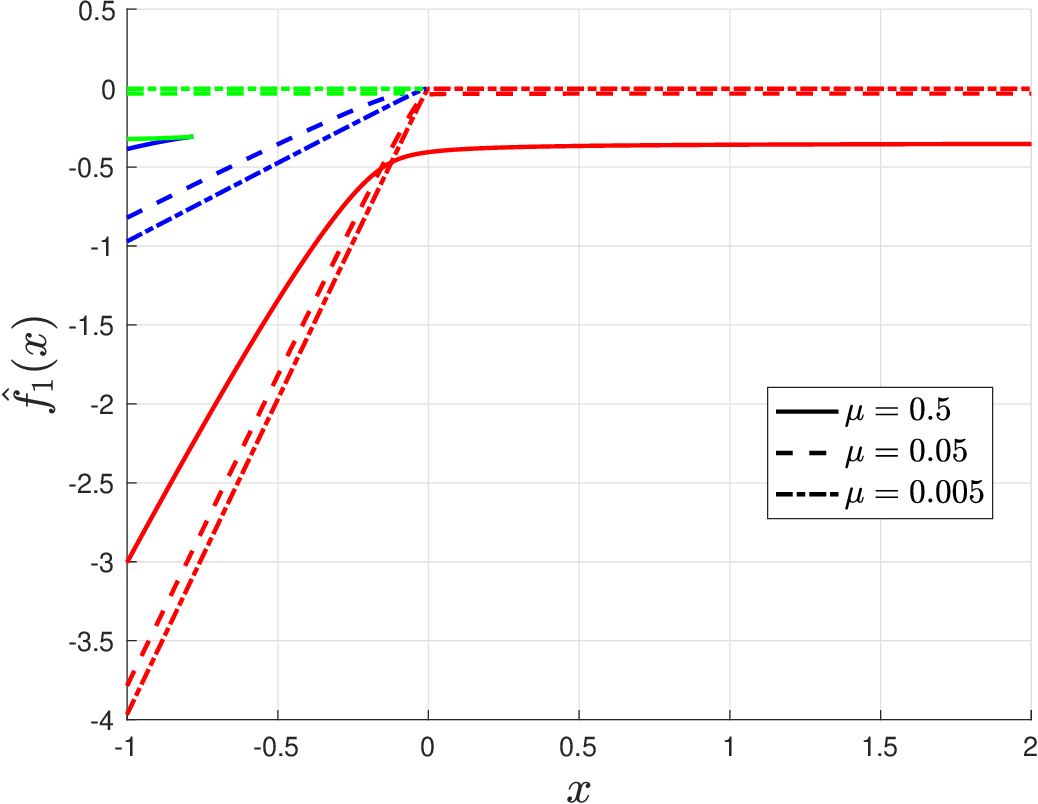}
    \caption{Solution maps (left) and smoothed master/subproblem problem objective (right) for Example~\ref{ex:curvature}.}
    \label{fig:QP_solution_maps_smooth}
\end{figure}

We finally note that with smoothing, subproblem~\eqref{eq:obj_smoothing_proj} satisfies LICQ and hence the KKT conditions hold at its local minima.
It is therefore of interests to study the trajectories of KKT points.
The notion of ``solution map", defined according to local minima in the above subsection, can be readily extended to KKT points.
Let us consider the notation $v_i^*$ defined in~\eqref{eq:v_i_definition}, and define the set
\begin{equation}
    \label{aeq:def_Si}
    \SSS_i = \{(x,v_i)\in\RR^{n_0}\times\RR^{n_{v_i}} :F_i(v_i;x,\mu)=0\}.
\end{equation}
Then a solution map with respect to KKT points of the barrier problem is a continuous map $v^*_i(\cdot;\mu):U\to \mathbb{R}^{n_{v_i}}$, if $(x,v^*_i(x;\mu))\in\SSS_i$ for all $x\in U$.

\section{Decomposition algorithm}
\label{sec:algorithm}

For simplicity of notation, we write only the $i$-th subproblem $\hat f_i(x)$ in the remaining of the paper.
Unless emphasized, the operations for $i$-th subproblem introduced are with respect to all subproblems for $i=1,\cdots,N$.

\subsection{Algorithm for the master problem}
\label{sec:master_algo}

We denote the master variables as $u$, which contains the primal master variables $x$, and possibly other quantities from the first stage (e.g., dual variables) depending on the algorithm choice.
For example, if an SQP method is applied to solving~\eqref{eq:intro_master}, where the constraint $c_0(x)\leq 0$ is associated with the dual variable $\lambda_0$ and the slack $s_0$, then $u=(x,\lambda_0,s_0)$.

We also assume in this paper the feasibility of subproblems parameterized by any feasible master variables.
This assumption is formally stated as Assumption~\ref{assump:full_recourse} in Section~\ref{sec:global_convergence}.
In a case where this assumption is violated, i.e., a subproblem is infeasible for some $x$, one can introduce slack variables to the constraints parameterized by $x$, and penalize the subproblem objective by the slacks.
With this, the subproblems are always feasible.
For more details, see, e.g., \cite[Chapter 4]{luo2023efficient}.

The general framework of our two-stage decomposition algorithm is summarized in Algorithm~\ref{alg:decomposition_algo_general}.
Importantly, we note that in Step~\ref{line:master_solver} whenever the master problem solver requires computing $\hat{f}_i(x)$ and its derivatives at some $x$, our algorithm calls a subproblem solver to solve~\eqref{eq:obj_smoothing} parameterized by $x$.

As illustrated in Section~\ref{sec:smoothing}, the algorithm iteratively solves a smoothed master problem~\eqref{eq:smooth_master} parameterized by the smoothing parameter $\mu$.
For a large value of $\mu$, there is no need to solve the master problem to high accuracy.  Instead, the master problem solver tolerance is tightened as $\mu\to0$, as stated in Step~\ref{line:master_solver}.
Here, let $\theta_0(\cdot;\mu):\RR^{n_u}\to\RR_+$ be a continuous optimality measure that is zero if and only if $u$ corresponds to a stationary point of the smoothed master problem \eqref{eq:smooth_master}.   
If an SQP solver is used, it makes sense to define $\theta_0$ in terms of the violation of the KKT conditions for~\eqref{eq:smooth_master} via
$\theta_0(u;\mu)=\|F_0(u;\mu)\|$, where 
\begin{equation}\label{eq:smooth_master_optcond}
  F_0(u;\mu) =
  \begin{pmatrix}
   \na_x f_0(x)+\na_x \hat{f}_i(x;\mu)+\na_x c_0(x)\lambda_0\\
        c_0(x) + s_0 \\
        s_0 \circ \lambda_0
  \end{pmatrix}
\end{equation}
with $\lam_0\geq0$ and $s_0\geq 0$. 

\SOUT{
Once the master problem solver generates an approximate stationary point $u^l$ satisfying $\theta_0(u^l;\mu^l)\leq c_0\mu^l$, the algorithm decreases the value of $\mu$ and computes an initialization point for the next iteration.}


\begin{algorithm}[t]
\caption{Two-stage decomposition algorithm} \label{alg:decomposition_algo_general}
\begin{algorithmic}[1]
\Require{Initial iterate $\tilde u^0$; initial smoothing parameter $\mu^0$; termination factor $c_0>0$.} 

\State Set $l\gets0$.
\State Starting from $\tilde u^l$, call a master problem solver 
to solve~\eqref{eq:smooth_master} with $\mu=\mu^l$ to find $u^{l}$ so that $\theta_0(u^l;\mu^l)\leq c_0\mu^l.$ 
\label{line:master_solver}
\State Choose $\mu^{l+1}\in(0,\mu^{l})$ (so that $\mu^l\to0$) and set $l\gets l+1$.\label{line:decay_mu}
\State Choose starting iterate $\tilde u^{l}$ for the next iteration and go to Step~\ref{line:master_solver}. 
\label{line:initialization}
\end{algorithmic}
\end{algorithm}

The master solver in Step~\ref{line:master_solver} can in principle be any off-the-shelf nonlinear programming algorithm.
To ensure {convergence from arbirary initialization and a fast local rate}, it typically requires a second-order method with a line search or trust region.
The only difference from a regular single-stage method is that it requires to call a subproblem solver to evaluate the function value and derivatives.
In Section~\ref{sec:global_convergence} and~\ref{sec:numerical}, we will showcase the convergence and implementation of a trust-region S$\ell_1$QP method as the master problem solver. 

A simple setting of the initialization in Step~\ref{line:initialization} would be $\tilde u^{l}\gets u^l$.
In the absence of LICQ for all subproblems (SLICQ), it was observed numerically that fast local convergence might not be achieved with such initialization \cite{tu2021two}.
In Section~\ref{sec:extrapolation}, we illustrate how to achieve the superlinear local convergence without SLICQ by using an extrapolation step.  To achieve this rate, the smoothing parameter must be decreased in a superlinear fashion, similar to \cite{wachter2006implementation}.

\subsection{Interior-point method for the subproblems}
\label{sec:subproblem_algo}

A subproblem solver is called in the master problem solver\SOUT{ in order to compute the function value, gradient, and Hessian of the master problem}.
In this paper, we choose a Newton-type method to be the subproblem solver, and we will prove in Section~\ref{sec:global_conv_fixed_mu} that this enables {the \SOUTr{global }convergence for nonconvex two-stage problems}.
Due to the log-barrier smoothing, it is natural to utilize an interior-point method.
In our experiments in Section~\ref{sec:numerical}, we use {\tt Ipopt} \cite{wachter2006implementation}.

\begin{sloppypar}
We also implement a \textit{warm start mechanism} whenever there is a change of the master variables parameterizing the subproblems.
To be specific, suppose at a master iterate $x$ the subproblem solver returns a stationary point $v_i^*(x;\mu)$ such that $F_i(v_i^*(x;\mu);x,\mu)=0$.
Whenever the master solver evaluates a trial point $x+\Delta x$, the subproblem solver initializes itself from $v_i^*(x;\mu)$ instead of a random or fixed initialization.
We will show in Section~\ref{sec:global_conv_fixed_mu} that warm start is a crucial component to achieve {the convergence from any initialization.}
\end{sloppypar}

\section{Global convergence analysis}
\label{sec:global_convergence}

This section concerns the global convergence properties of the proposed decomposition method, when both the master and subproblems are in general nonconvex.
{By global convergence, we aim to establish convergence of the iterates to stationary points of~\eqref{eq:intro_master} and~\eqref{eq:smooth_master} with arbitrary initialization.}
We present in Section~\ref{sec:global_conv_fixed_mu} the results \SOUT{for the master problem solver }with a fixed value of $\mu^l$, and in Section~\ref{sec:global_conv_changing_mu} for $\mu^l\to0$.

To facilitate the analysis, we have the following assumption on feasibility throughout the paper.
\begin{assumption}
\label{assump:full_recourse}
    For any $x$ sent to the subproblems by the master problem~\eqref{eq:intro_master}, the subproblem~\eqref{eq:intro_subproblem} parameterized by $x$ is feasible and the subproblem solver always returns a KKT point.
\end{assumption}

\subsection{Convergence for fixed values of the smoothing parameter}
\label{sec:global_conv_fixed_mu}
We first study the case where $\mu^l$ is fixed.
In light of this, we drop the dependency of $\mu^l$ and the index $l$ in the functions and variables of this subsection.

In Section~\ref{sec:hat_f_is_C2} we will discuss how to remedy the pathological behaviors of $\hat f_i(\cdot)$ and attain differentiability locally, crucial for convergence.
Then, we showcase how these results permit a global convergence proof in a trust-region S$\ell_1$QP framework.

\subsubsection{\boldmath Differentiability of \texorpdfstring{$\hat f_i(\cdot)$}{hatfi}}
\label{sec:hat_f_is_C2}
Recall from Section~\ref{sec:behaviors_nonconvex} that we have defined the notion of ``solution map" as the trajectory of KKT points/local minimizers for~\eqref{eq:obj_smoothing}.
In light of the necessary optimality conditions and the algorithm implemented, we consider in Section~\ref{sec:global_convergence} the solution maps with respect to KKT points.


First, we present a standard assumption on differentiability.
Depending on the smoothing technique (see \eqref{eq:obj_smoothing} and \eqref{eq:sol_smoothing}), we require different assumptions.
\begin{assumption}
\begin{sloppypar}
\label{assump:lipschitz}
    $f_0(\cdot)$ and $c_0(\cdot)$ are $C^1$ with locally Lipschitz continuous first derivatives.
    {If objective smoothing \eqref{eq:obj_smoothing} is implemented, $f_i(\cdot;\cdot)$ and $c_i(\cdot;\cdot)$ are $C^2$ with locally Lipschitz continuous second derivatives;
    if solution smoothing \eqref{eq:sol_smoothing} is implemented, $f_i(\cdot;\cdot)$ and $c_i(\cdot;\cdot)$ are $C^3$.}
\end{sloppypar}
\end{assumption}


The following definition helps us to refer to points at which the subproblem has a unique local stationary point for a fixed $x$; recall that $\SSS_i$ is defined in~\eqref{aeq:def_Si}.
\begin{sloppypar}
\begin{definition}
    A point $(x,v_i)\in \SSS_i$ is a nondegenerate stationary point, if $\nabla_{v_i} F_i(v_i;x)$ is nonsingular.  If $\nabla_{v_i} F_i(v_i;x)$ is singular, then $(x,v_i)\in \SSS_i$ is a degenerate stationary point.
\end{definition}
\end{sloppypar}



Next, we introduce the notion of \textit{``reference point"}, which helps to distinguish multiple solution maps locally.
Specifically, we call $v_i^*(\cdot;(\bar x, \bar v_i))$ a solution map with a reference point $(\bar x,\bar v_i)\in \mathbb{R}^{n_0+n_{v_i}}$ if $v_i^*(\bar x;(\bar x, \bar v_i))=\bar v_i$.
%
The following lemma considers a nondegenerate point $(\bar x,\bar v_i)\in\SSS_i$ of the subproblem as the reference point.
It shows there exists a neighborhood around $\bar x$ such that there is a unique smooth solution map, and the graph of the solution map is also unique in a neighborhood of $(\bar x, \bar v_i)$.

\begin{lemma}
\label{lem:unif_bound_IFT}
Suppose Assumption~\ref{assump:full_recourse} and~\ref{assump:lipschitz} hold.
Let $\bar x\in\RR^{n_0}$ and $\bar v_i\in\mathbb{R}^{n_{v_i}}$ such that $(\bar x,\bar v_i)\in \SSS_i$ is nondegenerate.  
Then there exists $r_1>0$ such that there is a unique $C^1$ solution map $v_i^*(\cdot;(\bar x, \bar v_i)):B(\bar x,r_1)\to\mathbb{R}^{n_{v_i}}$ on $B(\bar x,r_1)$ with the reference point $(\bar x,\bar v_i)$.

Further, there exists $r_2>0$ such that the graph of $v_i^*(\cdot;(\bar x, \bar v_i))$ is the unique smooth submanifold in $B((\bar x,\bar v_i),r_2)$ that contains $(\bar x, \bar v_i)$ and solves $F_i(\cdot;\cdot)=0$.


\end{lemma}

    

\begin{proof}
Since $F_i(\bar v_i;\bar x) =0$ and the Jacobian $\nabla_v {F_i}(\bar v_i;\bar x)$ is nonsingular, the first statement of the lemma follows directly from the implicit function theorem.
%

The second statement extends the uniqueness result to the $(x,v_i)$-space.
The inverse function theorem implies that the graph of $v_i^*(\cdot;(\bar x, \bar v))$ is a smooth embedded submanifold in $\mathbb{R}^{n_0+n_{v_i}}$, a Cartesian product between the $x$- and $v_i$-space. 
Then, the uniqueness of $v_i^*(\cdot;(\bar x, \bar v))$ and the nonsingularity of $\nabla_v {F_i}(\bar v_i;\bar x)$ guarantee that:
there exists $r_2>0$ and $r_2\leq r_1$ such that in $B((\bar x,\bar v_i),r_2)$, the graph of $v_i^*(\cdot;(\bar x, \bar v))$ is the unique smooth submanifold which is the zero level set of $F_i(\cdot;\cdot)$ containing $(\bar x,\bar v_i)$.
\end{proof}

The existence of $v^*_i(\cdot;(\bar x, \bar v_i))$ in Lemma~\ref{lem:unif_bound_IFT} enables us to study the local properties of solution maps, when a nondegenerate reference point is given.

Due to the existence of multiple solution maps when subproblems are nonconvex,
the function values $\hat f_i(x^k)$ that the first-stage algorithm ``sees'' might \SOUT{be somewhat erratic and }not correspond to a continuous function throughout the iterations.  
To make sure that the convergence properties of standard nonlinear optimization methods still hold, we make the following assumption on the subproblem solver.





\begin{assumption}
\label{ass:consistent_output}
    If $\bar x\in\RR^{{n_0}}$ and $\bar v_i\in\mathbb{R}^{n_{v_i}}$ such that $(\bar x,\bar v_i)\in \SSS_i$ is nondegenerate, then there exists $r_3>0$ such that the subproblem solver satisfies the following property:  

    Suppose that $v_i^{k}$ is the stationary point computed by the subproblem solver for a master problem iterate $x^{k}$ (i.e., $(x^{k},v_i^{k})\in \SSS_i$) so that $(x^{k},v_i^{k})\in B((\bar x,\bar v_i),r_3)$.  
    Then, when the subproblem solver is started from a starting point $v_i^{k}$ for an input $x^+ = x^{k}+p$ with $\|p\|\leq r_3$, the stationary point $v_i^+$ computed by the subproblem solver satisfies $(x^+,v_i^+)\in B((\bar x,\bar v_i),{r_2})${, where $r_2$ is defined in Lemma~\ref{lem:unif_bound_IFT}.}
    
\end{assumption}

This assumption essentially states that, once the input/output pair for the subproblem solver is sufficiently close to a nondegenerate stationary point of the subproblem and the master solver makes a sufficiently small update on the $x$-space, then the output returned by the subproblem solver is guaranteed to be close to the nondegenerate stationary point as well.
Lemma~\ref{lem:newton_region_of_attraction} below shows that a Newton-based interior-point solver for \eqref{eq:obj_smoothing} naturally satisfies this assumption.  This assumption is also always satisfied when there is a unique continuous solution map, e.g., when the subproblem is strictly convex. 


For a nondegenerate reference point $(\bar x,\bar v_i)\in\SSS_i$ and the corresponding solution map $v^*_i(\cdot;(\bar x,\bar v_i))$ given by Lemma~\ref{lem:unif_bound_IFT},  we define the local smoothed second-stage functions according to either solution or objective smoothing (similar to Section~\ref{sec:smoothing}):
\begin{equation}\label{eq:def_fhats}
\begin{aligned}
    \hat f_i^{\rm sol}(x;(\bar x,\bar v_i)) & = f_i(y^*_i(x;(\bar x,\bar v_i));x) \\
    \hat f_i^{\rm obj}(x;(\bar x,\bar v_i)) & = \hat f_i^{\rm sol}(x;(\bar x,\bar v_i))  - \mu\sum_j\ln(s^*_{ij}(x;(\bar x,\bar v_i)))
\end{aligned}
\end{equation}
\begin{sloppypar}
Note that these functions are well-defined in the neighborhood prescribed in Lemma~\ref{lem:unif_bound_IFT}.
Depending on the choice of the smoothing, we let $\hat f_i(x;(\bar x,\bar v_i)) = \hat f_i^{\rm obj}(x;(\bar x,\bar v_i))$ or $\hat f_i(x;(\bar x,\bar v_i)) =\hat f_i^{\rm sol}(x;(\bar x,\bar v_i))$.
\end{sloppypar}
\begin{proposition}
\label{prop:f_is_C2_single}
    Suppose Assumption~\ref{assump:full_recourse}, \ref{assump:lipschitz}, and \ref{ass:consistent_output} hold and $(\bar x,\bar v_i)\in \SSS_i$ is nondegenerate.
    Then, there exists $r>0$ such that for $(x,v_i)\in B((\bar x,\bar v_i),r)\cap \SSS_i$, the following hold:
    \begin{enumerate}[(i)]
        \item $(x,v_i)$ is on the graph of a unique $C^1$ solution map $v_i^*(\cdot;(\bar x,\bar v_i))$.
       
        \item 
        Let $x^+=x+p$ with $\|p\|\leq r$.
        Further let $\hat f_i(x)$ and $\hat f_i(x^+)$ be the values returned successively by the subproblem solver evaluated at $x$ and $x^+$.
        If $\hat f_i(x)=\hat f_i(x;(\bar x,\bar v_i))$, then $\hat f_i(x^+)=\hat f_i(x^+;(\bar x,\bar v_i))$.
        
         \item The function $\hat f_i(\cdot;(\bar x,\bar v_i))$ restricted to ${B(\bar x,r)}$ is $C^2$.

        
    \end{enumerate}
\end{proposition}

\begin{proof}
We take $r:=\min\{r_2,r_3\}$ with $r_2$ defined in Lemma~\ref{lem:unif_bound_IFT} and $r_3$ defined in Assumption~\ref{ass:consistent_output}.

\textit{Part (i):}
This claim follows directly from the second statement of Lemma~\ref{lem:unif_bound_IFT}.

\textit{Part (ii):}
By Assumption~\ref{ass:consistent_output}, the subproblem solver computes a stationary point $v_i^+$ for input $x^+$ with $(x^+,v_i^+)\in B((\bar x,\bar v_i),{r_2})\cap\SSS_i$.  
{Because the solution map solving $F_i(\cdot;\cdot)=0$ is unique in $B((\bar x,\bar v_i),r_2)$ by Lemma~\ref{lem:unif_bound_IFT}, both $(x,v_i)$ and $(x^+,v_i^+)$ are on the graph of $v_i^*(\cdot;(\bar x,\bar v_i))$.}
The claim then follows from the definition \eqref{eq:def_fhats}.

\textit{Part (iii):}
If objective smoothing is implemented,  since $v_i^*(\cdot;(\bar x,\bar v_i))$ is $C^1$ by Lemma~\ref{lem:unif_bound_IFT} and $\nabla \hat f_i^{\rm obj}(\cdot;(\bar x,\bar v_i))=\eta_i^*(\cdot;(\bar x,\bar v_i))$ by~\eqref{eq:na_hat_f}, the statement follows from the definition \eqref{eq:def_fhats}.  
If solution smoothing is implemented,  since $f_i$ and $c_i$ are $C^3$ from Assumption~\ref{assump:lipschitz}, the implicit function theorem further guarantees that $v_i^*(\cdot;(\bar x,\bar v_i))$ is $C^2$.
Then by applying the chain rule to \eqref{eq:def_fhats}, one has $\nabla \hat f_i^{\rm sol}(\cdot;(\bar x,\bar v_i))$ is $C^2$.
\end{proof}



\begin{remark}
    Proposition~\ref{prop:f_is_C2_single} and Assumption~\ref{ass:consistent_output} bridge the gap between the analysis of single- and two-stage optimization.
    As discussed, the function $\hat{f}_i$ is in general set-valued.
    However, with Assumption~\ref{ass:consistent_output} we have from Proposition~\ref{prop:f_is_C2_single} that in a neighborhood around a nondegenerate limit point, $\hat{f}_i$, as it is ``seen'' by the master problem solver, is $C^2$. 
    Therefore, we can safely borrow convergence results from smooth optimization which rely on applying Taylor's theorem near limit points.
\end{remark}

We conclude this subsection by showing that Assumption~\ref{ass:consistent_output} is very natural and can be satisfied by a line search Newton interior-point method and the most recent solution is used as a starting point, which is the \textit{warm start mechanism} introduced in Section~\ref{sec:subproblem_algo}.

\begin{algorithm}[t]
\caption{Basic Newton line search method} \label{alg:newton}
\begin{algorithmic}[1]
\Require{Initial primal-dual iterate $v_i^0$; first-stage variable $x$.}
\For{$j=0,1,2,\ldots$}
    \State Compute the Newton step from $\na_{v_i} F_i(v_i^j;x)^T\Delta v_i^j = -F_i(v_i^j;x)$.\label{line:Newton_step}
    \State Choose a step size $\beta^j\in(0,1]$ via line search.
    \State Update iterate: $v_i^{j+1} = v^j_i + \beta^j\Delta v_i^j$.
\EndFor
\end{algorithmic}
\end{algorithm}



\begin{lemma}[Newton's region of attraction with warm start]
\label{lem:newton_region_of_attraction}
    Suppose Assump\-tion~\ref{assump:full_recourse} and \ref{assump:lipschitz} hold and let $(\bar x,\bar v_i)\in \SSS_i$ be a nondegenerate stationary point.  
    Further assume that in Algorithm~\ref{alg:newton}, there exists $\beta_{\min}\in(0,1]$ so that $\beta^j\geq\beta_{\min}$ for all $j$.
    Then a subproblem solver implementing Algorithm~\ref{alg:newton} with warm start satisfies Assumption~\ref{ass:consistent_output}.
    Namely, there exists $r_3>0$ such that for $(x, v_i) \in B((\bar x,\bar v_i),r_3)\cap\SSS_i$ and $x^+= x+p$ for some $\|p\|\leq r_3$, the following holds:

    If Algorithm~\ref{alg:newton} initializes from inputs $v_i^0= v_i$ and $x^+$, it converges to a limit point $v_i^+$ so that $(x^+,v_i^+)\in B((\bar x,\bar v_i),{r_2})$.
\end{lemma}

\begin{proof}
    \begin{sloppypar}
    By Lemma~\ref{lem:unif_bound_IFT}, there exists a neighborhood $B((\bar x,\bar v_i),r_2)$ such that there is a unique $C^1$ solution map $v_i^*(\cdot;(\bar x,\bar v_i))$ and its graph is unique in $B((\bar x,\bar v_i),r_2)$.
    By Assumption~\ref{assump:lipschitz}, $\nabla F_i$ is locally Lipschitz continuous; see \eqref{eq:smooth_subprob_optcond}.
    Let $L$ be the Lipschitz constant of $\nabla F_i$ {in $\bar B((\bar x,\bar v_i),r_2)$}.
    {We then define $C_1:=\sup_{x\in {B(\bar x,\frac{r_2}{2})}}\|\nabla F_i(v_i^*(x;(\bar x,\bar v_i));x)^{-1}\|$, and $C:=\min\{\frac{1}{2LC_1},{\frac{r_2}{4}}\}$.}
    {Note that $C_1$ is finite since $B\!\left(\bar x, \tfrac{r_2}{2}\right)$ is a proper subset of $B(\bar x, r_2)$, within which the implicit function theorem guarantees that $\|\nabla F_i^{-1}\|$ is uniformly bounded, by Lemma~\ref{lem:unif_bound_IFT}.}
    Since $B(\bar x,r_2)$ is bounded and $v_i^*(\cdot;(\bar x,\bar v_i))$ is $C^1$, there exists a constant $r_4>0$ such that {$r_4\leq \frac{r_2}{2}$ and} $\|v_i^*(x_1;(\bar x,\bar v_i))-v_i^*(x_2;(\bar x,\bar v_i))\|\leq C$ for any $x_1,x_2\in B(\bar x,r_4)$.
    \end{sloppypar}

    We let $r_3=\frac{r_4}{2}$. Since $r_3<r_4< r_2$, $v_i=v_i^*(x;(\bar x,\bar v_i))$ {for any $(x, v_i) \in B((\bar x,\bar v_i),r_3)\cap\SSS_i$.}
    {Then for any }$\|p\|\leq r_3$ and $x^+=x+p$, we have {$x^+\in B(\bar x,2r_3)$}.
    
    It remains to show: starting from $v_i^0=v_i$ to solve the subproblem parameterized by $x^+$, Algorithm~\ref{alg:newton} converges to $v_i^* := v_i^*(x^+;(\bar x,\bar v_i))$.
    {With this claim, $\|(x^+,v_i^+)-(\bar x,\bar v_i)\|{\leq}\|x^+-\bar x\|+\|v_i^*-\bar v_i\|\leq 2r_3+C\leq r_2$, which finishes the proof.}

    {Now to show the claim above}, we first analyze the decrease given by a full Newton step.
    At {the} $j$-th iteration, let $\hat{v}_i^{j+1}:=v_i^j+\Delta v^j$ where $\Delta v^j$ is the full Newton step.
    Then, following the procedure of the convergence analysis for a standard Newton's method (see, e.g., \cite{nocedal1999numerical} proof of Theorem 3.5), one has that for {the} $j$-th iteration
    \begin{equation}
    \label{ineq:newton_decay}
        \|\hat{v}_i^{j+1}-v_i^*\|\leq  L\|\nabla F_i(v_i^*;x^+)^{-1}\|\|v_i^j-v_i^*\|^2 \leq \frac{1}{2C}\|v_i^j-v_i^*\|^2
    \end{equation}
    {whenever $(x^+,v_i^j)\in B((\bar x,\bar v_i),r_2)$.}
    {Here, the second inequality holds because $x^+ \in B(\bar x, 2r_3) \subset B\left(\bar x, \tfrac{r_2}{2}\right)$, which ensures that $\|\nabla F_i(v_i^*; x^+)^{-1}\| \le C_1$.}
    Note that since $x,x^+\in B(x,r_4)$, we have $\|v_i^0-v_i^*\|=\|v_i^*(x;(\bar x,\bar v_i))-v_i^*(x^+;(\bar x,\bar v_i))\| \leq C$,
    {and we also have $\|v_i^1-v_i^*\|=\|v_i^0+\beta^0\Delta v_i^0-v_i^*\|\leq  \beta^0\|\hat{v}_i^1-v_i^*\| +(1-\beta^0)\| v_i^0-v_i^* \|$.}
    By induction, we can prove $\|v_i^{j+1}-v_i^*\|\leq (1-\frac{\beta_{\min}}{2})\|v_i^j-v_i^*\|$ for $j=0,1,2,\cdots$.
    {Due to the space limits, we include the detailed steps of the induction in Appendix~\ref{appendix:supp_proofs}.}
    Finally since $1-\frac{\beta_{\min}}{2}<1$ for all $j$, $\|v_i^j-v_i^*\|\to 0$, and $v_i^j\to v_i^*$, as desired.
\end{proof}

We remark that the assumption of the existence of $\beta_{\min}$ is not strong.  In fact, most algorithms for solving \eqref{sec:numerical} include procedures like second-order corrections steps \cite{nocedal1999numerical} that ensure full steps can be taken when the iterate is close to a nondegenerate solution.

\subsubsection{Global convergence of a trust-region SQP framework}
\label{sec:SQP_global_convergence_fixed_mu}

In this subsection, we illustrate how Proposition~\ref{prop:f_is_C2_single} enables the adaptation of an existing global convergence proof for a general nonlinear programming algorithm\SOUT{ that is used to solve the master problem in Step~\ref{line:master_solver} of Algorithm~\ref{alg:decomposition_algo_general}} to our two-stage setting.  As a particular example, we consider the trust-region S$\ell_1$QP algorithm presented as Algorithm~11.1.1 in \cite{conn2000trust} for the solution of the master problem.

Algorithm~11.1.1 in \cite{conn2000trust} can be used to minimize the exact $\ell_1$-penalty function
\begin{equation}\label{eq:phi}
    \phi(x)= f_0(x)+\hat f_i(x)+\pi\norm{[c_0(x)]^+}_1.
\end{equation}
Here, $\pi>0$ is a penalty parameter, and it is well known that, under standard assumptions, one can recover a local optimum of \eqref{eq:smooth_master} from a local minimum of $\phi$ if $\pi$ sufficiently large; see, e.g., \cite[Theorem 14.5.1]{conn2000trust}.
To keep matters simple, we are not concerned here with finding a suitable value of $\pi$.  \SOUT{Instead, we assume that $\pi$ are been chosen sufficiently large.}


As a trust-region method, at an iterate $x^k$, we define a local model of $\phi$ as
\begin{equation}\label{eq:phi_model}
\begin{aligned}
     m(x^k,H^k,p):=\,& f_0(x^k)+\hat f_i(x^k)+\nabla [f_0(x^k)+\hat f_i(x^k)]^Tp+\half p^T H^k p + \\
     & \pi \norm{[c_0(x^k)+\na c_0(x^k)^T p]^+}_1
\end{aligned}
\end{equation}
where the symmetric matrix $H^k$ typically attempts to capture second-order curvature information, 
and compute a trial step $p^k$ as an optimal solution of
\begin{equation}\label{eq:tr_subproblem}
    \min_{p\in\RR^{n_0}} m(x^k,H^k,p),\,\,{\rm s.t.}\,\, \|p\|\leq \Delta^k.
\end{equation}
This algorithm is called S$\ell_1$QP because \eqref{eq:tr_subproblem} is equivalent to the following QP: 
\begin{equation} \label{eq:sqp_qp}
\begin{split}
    \min_{p,t} \quad & \nabla [f_0(x^k)+\hat f_i(x^k)]^Tp+\frac{1}{2}p^TH^kp+\pi\sum_{j}t_j \\
        {\rm s.t.} \quad & \nabla c_0(x^k)^Tp+c_0(x^k)\leq t,\,\,t\geq 0,\,\,\|p\|\leq \Delta^k,
\end{split}
\end{equation}
where $t\in\mathbb{R}^{m_0}$.
When second derivatives are available, $H^k$ is set as $\na_{xx}^2\LLL(x^k,\lambda^k_0)$, the Hessian of the Lagrangian for the master problem, where $\lambda^k_0$ is an estimate of the dual variables corresponding to $c_0$.
The full algorithm (Algorithm 11.1.1 in \cite{conn2000trust}) is stated in Algorithm~\ref{alg:trust_region}.


\begin{algorithm}[t]
\caption{Two-stage trust-region master solver} \label{alg:trust_region}
\begin{algorithmic}[1]
\Require{Initial master and subproblem iterate $x^0$, $v^0_i$;
trust-region radius parameters $0<\Delta^0\leq \bar \Delta$;
penalty parameter $\pi>0$;
trust-region parameters $\eta_1, \eta_2, \gamma_1, \gamma_2,\gamma_3$ satisfying
$0<\eta_1\leq\eta_2<1 \text{ and } 0<1/\gamma_3\leq\gamma_1\leq\gamma_2<1<\gamma_3.$
}
\State $k\gets 0$.
\State Call the subproblem solver to find a stationary point $(x^k,\tilde v_i^k)\in\SSS_i$ of \eqref{eq:obj_smoothing}, using $v_i^k$ as starting point.  It returns $\hat f_i(x^k)$ and $\nabla\hat f_i(x^k)$ based on $\tilde v_i^k$.
\For{$k=0,1,2,\ldots$}
    \State Choose $H^k$ and solve \eqref{eq:tr_subproblem} to get $p^k$.
    \State If $p^k=0$ terminate and return $x^k$ as a stationary point.
    \State Call the subproblem solver to find a stationary point $(x^k,\tilde v_i^k)\in\SSS_i$ of \eqref{eq:obj_smoothing}, using $v_i^k$ as starting point.  It returns $\hat f_i(x^k)$ and $\nabla\hat f_i(x^k)$ based on $\tilde v_i^k$.\label{bline:subproblem_solve_in_master}
    \State Compute the performance ratio
    $\rho^k=\frac{\phi(x^k)-\phi(x^k+p^k)}{m^k(x^k,0)-m^k(x^k,p^k)}.$
    \State Update the first-stage iterate
    $x^{k+1}\leftarrow\begin{cases}
        x^k+p^k,\quad & \text{if }\rho^k\geq \eta_1, \\
        x^k,\quad & \text{if }\rho^k< \eta_1.
    \end{cases}$ \label{bline:master_iter_update}
    \State Update the subproblem starting point
    $ v_i^{k+1}\leftarrow\begin{cases}
        \tilde v_i^k,\quad & \text{if }\rho^k\geq \eta_1, \\
         v_i^k,\quad & \text{if }\rho^k< \eta_1.
    \end{cases}$ 
    \State Update the trust-region radius. Set
    $$\Delta^{k+1}\in \begin{cases}
        [\gamma_3\Delta^k,\bar\Delta],\quad 
 & \text{if }\rho^k\geq \eta_2, \\
 [\gamma_2\Delta^k,\Delta^{k}],\quad 
 & \text{if }\rho^k\in [\eta_1,\eta_2), \\
 [\gamma_1\Delta^k,\gamma_2\Delta^{k}],\quad 
 & \text{if }\rho^k< \eta_1.
    \end{cases}$$
    \State $k\gets k+1$.
\EndFor
\end{algorithmic}
\end{algorithm}

Note that we explicitly track a warm start point $v_i^k$ for the subproblem solver, which is updated whenever an iterate is accepted.\SOUT{When a second-order subproblem solver is implemented, Assumption~\ref{ass:consistent_output} holds so that Proposition \ref{prop:f_is_C2_single} can be applied.}
In practice, there is no need for the master problem solver to store $v_i^k$. 
Instead, it can signal the subproblem solver to replace the starting point by the most recent solution whenever needed.
%

%
For our discussion here we assume that $p^k$ is an exact optimal solution of~\eqref{eq:tr_subproblem}, but this requirement can be relaxed, as long as $p^k$ provides at least as much decrease in $m^k$ as the Cauchy step; see \cite[Eq.~(11.1.9)]{conn2000trust} for its definition. 

Following from~\eqref{eq:def_fhats} and Proposition~\ref{prop:f_is_C2_single}, given a reference point $(\bar x,\bar v_i)$ we define
\begin{equation}
    \phi(x;(\bar x,\bar v_i)):=f_0(x)+\hat{f}_i(x;(\bar x,\bar v_i))+\pi\|[c_0(x)]^+\|_1
\end{equation}
so that $\phi(\cdot;(\bar x,\bar v_i))$ is a well-defined function in a neighborhood of $\bar x$.
Our goal is to show that any limit point $(x^\infty,v_i^\infty)$ of the iterate sequence corresponds to a stationary point of $\phi(\cdot;(x^\infty,v_i^\infty))$.
Since $\phi(\cdot;(\bar x,\bar v_i))$ is nonsmooth, we consider the stationary measure \cite[Eq. (11.1.4)]{conn2000trust}
\begin{equation}
    g(x):=\arg\min_{g\in \partial \phi(x;(\bar x,\bar v_i))}\|g\|,
\end{equation}
where 
$\partial \phi(x)
= \left\{g\in\RR^{n_0}:g^Td\leq \lim_{t\downarrow 0}\frac{\phi(x+td)-\phi(x)}{t}\text{ for all }d\in\RR^{n_0}\right\}
$ defines the subdifferential.
%
Theorem~\ref{thm:global_conv} is essentially \cite[Theorem 11.2.5]{conn2000trust}, stating that every limit point of the iterate sequence is a stationary point of $\phi(\cdot;(\bar x,\bar v_i))$ at which $g$ is zero.

Our proof of Theorem~\ref{thm:global_conv} argues that the original proof of Theorem 11.2.5 in \cite{conn2000trust} can still be applied because Proposition~\ref{prop:f_is_C2_single} provides the required smoothness of the problem functions.
More specifically, under our assumptions, it is possible to establish the following lemma, which corresponds to Lemma~11.2.3 in \cite{conn2000trust}.

\begin{lemma}\label{lem:rhoeta}
    Suppose Assumptions~\ref{assump:full_recourse},~\ref{assump:lipschitz},~\ref{ass:consistent_output} hold, $\{H^k\}$ is bounded, and let $(\bar x,\bar v_i)\in\SSS_i$ be a nondegenerate stationary point.
    Further assume that $\bar x$ is not a stationary point of $\phi(\cdot;(\bar x,\bar v_i))$. Let $\eta\in(0,1)$.
    Then there exists $r_5>0$ and $\Delta^{\max}>0$ so that for any iterate $(x^k,v_i^k)$ of Algorithm~\ref{alg:trust_region} in $B((\bar x,\bar v_i),r_5)$ and a step $p^k$ with $\|p^k\|\leq\Delta^k\leq \Delta^{\max}$ we have
    \begin{equation}\label{eq:localratioOK}
    \rho^k=\frac{\phi(x^k;(\bar x,\bar v_i))-\phi(x^k+p^k;(\bar x,\bar v_i))}{m^k(x^k,0)-m^k(x^k,p^k)}\geq\eta.
\end{equation}
\end{lemma}
\begin{proof}
    By Proposition~\ref{prop:f_is_C2_single} (iii), there exists a neighborhood $B((\bar x,\bar v_i),r)$ such that $\hat f_i(\cdot)=\hat f_i(\cdot;(\bar x,\bar v))$ is $C^1$ with Lipschitz gradient in ${B(\bar x,r)}$. 
    {By Proposition~\ref{prop:f_is_C2_single} (ii), let $\Delta^{\max}\leq r$ then $\hat{f}_i(\cdot)$ evaluates $x^k$ and $x^k+p^k$ on the same solution map.}
    This means that Taylor's theorem can be applied at both points. As a consequence, the proofs of Theorem 11.5.1 and Lemma 11.2.1 in \cite{conn2000trust} are still valid.  They imply that there exists $C_2>0$ so that
    $$|\phi(x^k+p^k;(\bar x,\bar v_i))-m^k(x^k,p^k)|\leq C_2\|p^k\|^2$$
    for any $x^k\in {B(\bar x,r)}$, assuming that $\{H^k\}$ is bounded. 
    This captures the approximation accuracy of $m^k$ to $\phi$.

    Since we assume in Algorithm~\ref{alg:trust_region} that $p^k$ is the optimal solution for~\eqref{eq:tr_subproblem}, it is also a Cauchy step \cite[Chapter 6, 11]{conn2000trust}: there exists $\delta>0$, $r_6>0$ and $\kappa\in(0,1)$ such that for ${x^k}\in B(\bar x,r_6)$,
    \begin{equation}\label{ineq:Cauchy_decrease}
        m^k(x^k,0)-m^k(x^k,p^k)\geq \kappa \|g(x^k)\|\min\{\delta,\Delta^k\}.
    \end{equation}
    
    Since $\bar x$ is not stationary, there exists $r_7>0$ and $\epsilon_1>0$ such that $\|g(x^k)\|\geq \epsilon_1$ for all $x^k\in {B(\bar x,r_7)}$; see, e.g., \cite[Lemma 11.1.2]{conn2000trust}. 
    By $m^k(x^k,0)=\phi(x^k)$,~\eqref{ineq:Cauchy_decrease}, and $\|p^k\|\leq\Delta^k$, we have that there exists $r_8\leq \min\{r,r_6,r_7\}$ and $\Delta_1^{\max}\leq\delta$ such that for any $x\in {B(\bar x,r_8)}$ and $\Delta^k\leq \Delta_1^{\max}$, 
    \begin{equation}
    \begin{aligned}
        \rho^k & =1-\frac{\phi(x^k+p^k;(\bar x,\bar v_i))-m^k(x^k,p^k)}{m^k(x^k,0)-m^k(x^k,p^k)}  \geq 1-\frac{|\phi(x^k+p^k;(\bar x,\bar v_i))-m^k(x^k,p^k)|}{\kappa\|g(x^k)\|\min\{\delta,\Delta^k\}} \\
        & \geq 1-\frac{C_2\|p^k\|^2}{\kappa\epsilon_1\min\{\delta,\Delta^k\}} = 1-\frac{C_2\|p^k\|^2}{\kappa\epsilon_1\Delta^k}  \geq 1-\frac{C_2\|p^k\|}{\kappa\epsilon_1}.
    \end{aligned}
    \end{equation}
    Finally, one can pick $r_5\leq r_8$ and $\Delta^{\max}\leq \min\{\Delta^{\max}_1,\frac{\kappa\epsilon_1(1-\eta)}{C_2}\}$ so that for any $x\in {B(\bar x,r_5)}$ and $\|p^k\|\leq \Delta^k\leq \Delta^{\max}$, $\rho^k\geq \eta$. 
    This finishes the proof.
\end{proof}

\begin{theorem}
\label{thm:global_conv}
    Suppose Assumption~\ref{assump:full_recourse}, \ref{assump:lipschitz}, and \ref{ass:consistent_output} hold and $\{H^k\}$ is bounded. 
    Let $(x^\infty,v_i^\infty)$ be a limit point of the sequence $\{(x^k,v_i^k)\}$ generated by Algorithm~\ref{alg:trust_region}, with the merit function \eqref{eq:phi} and the model \eqref{eq:phi_model}.
    If $\nabla_{v_i}F_i(v_i^\infty;x^\infty)$ is nonsingular, then $x^\infty$ is a stationary point of $\phi(\cdot;(\bar x,\bar v_i))$. 
\end{theorem}
This claim was originally proven in \cite{conn2000trust}
assuming all functions in \eqref{eq:smooth_master} are globally continuously differentiable.
In our setting, however, we know from Proposition~\ref{prop:f_is_C2_single} that $\hat f_i$, under Assumption~\ref{assump:lipschitz} and~\ref{ass:consistent_output}, is only guaranteed to be differentiable in a neighborhood $B((x^\infty,v_i^\infty),r)$ of any nondegenerate limit point $(x^\infty,v_i^\infty)\in\SSS_i$.
In the following we argue that the proof in \cite{conn2000trust} nevertheless also applies in our setting. 

\begin{proof}[Proof of Theorem~\ref{thm:global_conv}]
In \cite{conn2000trust}, Theorem 11.2.5 is proved by contradiction:  Suppose there exists a limit point $(x^\infty,v_i^\infty)$ of $\{(x^k,v_i^k)\}$ so that $x^\infty$ is not a first-order critical point.  Let $\{(x^k,v_i^k)\}_{\KKK}$ be a subsequence that converges to $(x^\infty,v_i^\infty)$.  Since $(x^k,v_i^k)$ does not change in subsequent iterations of Algorithm~\ref{alg:trust_region} when a new iterate is not accepted, one can assume that $\KKK$ includes only successful iterations in which $\rho^k\geq\eta_1$.

Then, by Lemma 11.2.4 in \cite{conn2000trust} (which we discuss in the next paragraph), there exists a threshold $\Delta_{\min}>0$ so that $\Delta_k\geq\Delta_{\min}$ for all $k\in\KKK$.
Following standard arguments and using \eqref{ineq:Cauchy_decrease}, this implies that $\phi(x^k;(x^\infty,v_i^\infty))-\phi(x^{k+1};(x^\infty,v_i^\infty))\geq c_\phi$ for some $c_\phi>0$ for all $k\in\KKK$.
And because $\phi(x^k;(x^\infty,v_i^\infty))$ is monotonically decreasing and bounded below, this yields the desired contraction.

What remains to establish is Lemma 11.2.4,
which we state here in a weaker form that suffices for the proof of Theorem 11.2.5:  Given the sequence $\{(x^k,v_i^k)\}_{\KKK}$ from above, there exists $\Delta_{\min}>0$ so that $\Delta_k\geq\Delta_{\min}$ for all $k\in\KKK$.
The proof of Lemma 11.2.4 {makes repeated use of the fact that there exists} a neighborhood $N$ around $x^*$ and $\Delta^{\max}>0$ so that~\eqref{eq:localratioOK} holds whenever $k\in\KKK$, $x^k\in N$, and $\Delta^k\leq\Delta^{\max}$.
This fact has been proven as Lemma~\ref{lem:rhoeta} above.
%
\end{proof}

\begin{remark}
If $x^{\infty}$ is feasible and LICQ holds at $x^{\infty}$, then there exists multipliers $\lambda_0^{\infty}$ so that $F_0(x^\infty,\lambda_0^\infty;\mu)=0$; see Theorem 17.4 in \cite{nocedal1999numerical}.

\end{remark}

\subsubsection{Nondegeneracy of the limit point}
\label{sec:nondegenerate_justification}
An assumption frequently made in this paper is that a point (an iterate or more crucially, a limit point of iterates) $(x,v_i)$ is nondegenerate, i.e., $\nabla_{v_i}F_i(v_i;x)$ is nonsingular.  While this may seem overly restrictive from a first impression, we argue next that it is rather benign and holds in many scenarios.

Recalling $F_i$ from \eqref{eq:smooth_subprob_optcond}, it can be shown by block elimination\SOUT{(recall that $s_i, \lam_i>0$)} that $\nabla_{v_i}F_i$ is nonsingular if and only if the symmetric indefinite matrix
\[ 
\begin{bmatrix}
        W_i & 0  & A_i^T \\
        0 &\Sigma_i & I \\
        A_i & I  & 0
    \end{bmatrix}
\]
with $W_i=\na_{y_iy_i}^2\LLL_i(v_i;x,\mu)$, $\Sigma_i={\text{diag}(s_i)^{-1}\text{diag}(\lambda_i)}$, and $A_i=\na_{y_i}c_i(y_i; x)^T$ is nonsingular, or equivalently, {the reduced Hessian} $(W_i+A_i^T\Sigma_iA_i)$ is nonsingular.
{
This nonsingularity ensures sufficient second-order information for solving the system along feasible directions.  
The role of the reduced Hessian is well established; see~\cite[Section~16--17]{nocedal1999numerical} for details.  
The condition clearly holds when $W_i$ is positive definite, or when $W_i$ is positive semi-definite, $m_i \geq n_i$, and $A_i$ has full column rank.  
These are sufficient but not necessary conditions and can also hold in nonconvex subproblems.  
For example, subproblems with convex objectives and nonconvex constraints (where the constraint curvature dominates the objective) may still satisfy this at a limit point.
Otherwise, failure of nonsingularity can result in degeneracy, such as intersecting solution maps; see Example~\ref{ex:curvature}.
}
In such cases, one could consider regularization techniques that convexify the problem so that $W_i$ becomes nonsingular, e.g., Tikhonov regularization \cite{borges2021regularized} and Moreau envelopes \cite{liu2024moreau}.

Finally we remark that a limit point of the primal-dual sequence $\{(x^k,v_i^k)\}$ might not exist, even if the primal variables $x^k$, $y_i^k$, and $s_i^k$ converge.  
We have seen in Figure~\ref{fig:linear_solution_maps_smooth} of Section~\ref{sec:behaviors_nonconvex} that as $x\to1$, $s_1$ converges to zero.  In that case, due to the complementarity $s_i\circ \lam_i=\mu e$, $\lam_i^k$ would converge to infinity and a limit point $(\bar x, \bar v_i)$ would not exist.  
However, if objective smoothing is used, then $\hat f_i(x)$ includes the $\log$-barrier term in \eqref{eq:obj_smoothing},  which goes to infinity when $s^k\to 0$; see, e.g., the right panel of Figure~\ref{fig:linear_solution_maps_smooth}.  
Therefore, a minimum-seeking first-stage algorithm would automatically be repelled from the degeneracy caused by vanishing slacks.

\subsection{Convergence for decreasing smoothing parameters}
\label{sec:global_conv_changing_mu}

In this subsection we study the global convergence\SOUT{ of Algorithm~\ref{alg:decomposition_algo_general}.
Specifically, we study the convergence} of the master problem iterates, as $\mu^l\to 0$. 

We begin by recalling the fact that the original two-stage problem \eqref{eq:intro_master} and \eqref{eq:intro_subproblem} is equivalent to the undecomposed single-stage optimization problem:
\begin{equation} \label{eq:undecomp_Obj}
\begin{split}
    \min_{ x\in \mathbb{R}^{n_{0}},\{y_i \in \mathbb{R}^{n_{i}} \}} \quad & f_0(x) + \sum_{i=1}^N f_i(y_i;x) \\
        {\rm s.t.} \quad & c_0(x) \leq 0 \\
        & c_i(y_i;x) \leq 0, \quad \forall i = 1, \dots, N.
\end{split}
\end{equation}
When objective smoothing is implemented, the smoothed two-stage optimization problem can also be equivalently written as:
\begin{equation}
\label{eq:undecomp_Obj_smoothed}
\begin{split}
    \min_{ x,\{y_i,s_i,\tilde x_i \}} \quad & f_0(x) + \sum_{i=1}^N \left[f_i(y_i;\tilde x_i)-\mu\sum_{j}\ln(s_{ij})\right] \\
        {\rm s.t.} \quad & c_0(x) \leq 0 \\
        & c_i(y_i;\tilde x_i)+s_i= 0, \quad \forall i = 1, \dots, N, \\
        & \tilde x_i-x=0,\quad \quad \qquad\forall i = 1, \dots, N.
\end{split}
\end{equation}

We will next show that any limit point generated by Algorithm~\ref{alg:decomposition_algo_general} with objective smoothing, i.e., any limit point of KKT points for~\eqref{eq:undecomp_Obj_smoothed} as $\mu\to0$, is a KKT point for \eqref{eq:undecomp_Obj}.

Let $u^l=(x^l,s_0^l,\lam_0^l)$ be the sequence generated by Algorithm~\ref{alg:decomposition_algo_general}.  
Furthermore, for each $x^l$, let $v_i^l=(y_i^*(x^l;\mu^l), \tilde x_i^*(x^l;\mu^l), s_i^*(x^l;\mu^l), \lambda_i^*(x^l;\mu^l), \eta_i^*(x^l;\mu^l))$
be the corresponding primal-dual solution of \eqref{eq:obj_smoothing_proj} that the subproblem solver generated.  





\begin{theorem}
    Suppose Assumption~\ref{assump:full_recourse} and \ref{assump:lipschitz} hold, and Algorithm~\ref{alg:decomposition_algo_general} generates a sequence of iterates $\{u^l\}$ with corresponding subproblem solutions $\{v_i^l\}$.
    Let $(u^*,v_i^*)$ be a limit point of $\{(u^l,v_i^l)\}$.  Then 
    $(x^*,y_i^*)$ is a KKT point of \eqref{eq:undecomp_Obj}.
    
\end{theorem}
\begin{proof}
    To simplify the notation, we assume without loss of generality that $N=1$.  Then the KKT conditions for \eqref{eq:undecomp_Obj} are
    \begin{equation}\label{eq:kktlarge}
    \begin{split}
        \na f_0(x) + \na_{x}  f_1(y_1;x) + \na c_0(x)\lam_0 + \na_{x} c_1(y_1;x)\lam_1 & = 0 \\
        \na_{y_1}  f_1(y_1;x) + \na_{y_1} c_1(y_1;x)\lam_1 & = 0 \\
        c_0(x) \leq 0 & \perp \lam_0 \geq 0 \\
        c_1(y_1;x) \leq 0 & \perp \lam_1 \geq 0.
    \end{split}
    \end{equation}


    Let $\{(u^{l_i},v_1^{l_i})\}$ be a subsequence converging to $(u^*,v_1^*)$.
    In Step~\ref{line:master_solver} of Algorithm~\ref{alg:decomposition_algo_general}, for each internal iterate, the master problem solver calls the subproblem solver to obtain $\na \hat f_1(x^l;\mu^l)$ for its solution $x^l$.  
    For the objective smoothing, this quantity is computed by \eqref{eq:na_hat_f}, i.e., $\na \hat f_1(x^{l_i};\mu^{l_i})=-\eta_1^*(x^{l_i};\mu^{l_i})$.  
    Taking the limit for the subsequence, we obtain that $\na \hat f_1(x^{l_i};\mu^{l_i}) \to -\eta_1^{*}$, where $\eta_1^*$ is a subvector in $v_1^*$.  
    Substituting \eqref{eq:na_hat_f} into $F_0(u;\mu)$ defined in \eqref{eq:smooth_master_optcond} and noting that $\|F_0(u^{l_i};\mu^{l_i})\|\to 0$ by the criterion in Step~\ref{line:master_solver}, we see that
    \begin{equation}
    \label{eq:cond_part1}
    \begin{split}
        \na f_0(x^*) - \eta_1^* + \na c_0(x^*)\lam^*_0 & = 0 \\
        c_0(x^*) \leq 0 & \perp \lam^*_0 \geq 0.
    \end{split}
    \end{equation}
    On the other hand, $v_1^{l_i}$ satisfies the KKT conditions for the subproblem \eqref{eq:obj_smoothing_proj}, i.e., $F_1(v_1^{l_i};x^{l_i},\mu^{l_i})=0$ with $s^{l_i}_1,\lam^{l_i}_1 \geq 0$.  Taking the limit yields
    \begin{equation}
    \label{eq:cond_part2}
    \begin{split}
        \na_{y_1} f_1(y_1^*;\tilde x_1^*) +\na_{y_1} c_1(y^*_1;\tilde x_1^*)\lam_1^* & = 0 \\
        \na_{\tilde x_1} f_1(y^*_1;\tilde x_1^*) + \na_{\tilde x_1} c_1(y^*_1;\tilde x_1^*)\lam_1^* + \eta_1^* & = 0 \\
        c_1(y^*_1;\tilde x_1^*) + s^*_1  & = 0\\
        \tilde x_1^* - x^* & = 0 \\
        s^*_1\geq 0 & \perp \lam^*_1 \geq 0.
    \end{split}
    \end{equation}
    Combining~\eqref{eq:cond_part1},~\eqref{eq:cond_part2}, and eliminating $\eta_1^*$, $\tilde x_1^*$, and $s_1^*$ yields \eqref{eq:kktlarge}.
\end{proof}

\section{Fast local convergence}
\label{sec:local_convergence}

In this section we present a variation of Algorithm~\ref{alg:decomposition_algo_general} that exhibits a superlinear local convergence rate under standard nondegeneracy assumptions.  
Our discussion here only pertains to objective smoothing \eqref{eq:obj_smoothing}, and we also assume that Hessians can be computed.
We assume that the master solver is a second-order SQP solver with $H^k$ in~\eqref{eq:phi_model} being the exact Hessian, and the subproblem solver is a Newton-based interior-point method (Algorithm~\ref{alg:newton}).
We let $N=1$ for simplicity, and we make the following assumption throughout this section.
\begin{assumption}
\label{assump:local}
Let $w^*=(u^*,v_1^*)$ be the primal-dual solution corresponding to a local minimum of the undecomposed problem~\eqref{eq:undecomp_Obj} that satisfies the second-order sufficiency conditions \cite[Theorem 12.6]{nocedal1999numerical} and strict complementarity.  
Further suppose that $f_0(\cdot)$, $f_1(\cdot;\cdot)$, $c_0(\cdot)$, and $c_1(\cdot;\cdot)$ are $C^2$ and have locally Lipschitz continuous second derivatives at $(x^*,y_1^*)$. 
Finally, LICQ holds at $(x^*,y_1^*)$ for~\eqref{eq:undecomp_Obj}.
\end{assumption}

We remark that the literature on two-stage optimization \cite{demiguel2008decomposition,tammer1987application} achieves a superlinear rate only by assuming SLICQ, which requires LICQ for all (nonsmoothed) subproblems.
This is quite restricted, and we instead consider a more general setting, where LICQ holds only for the undecomposed problem~\eqref{eq:undecomp_Obj}.
For instance, Example~\ref{ex:bilinear} does not satisfy SLICQ, but satisfies Assumption~\ref{assump:local} at the global optimum.

\subsection{Algorithm variant with the extrapolation step}
\label{sec:extrapolation}

Recall the undecomposed formulation of the barrier problem when $N=1$: 
\begin{equation}\label{eq:undecomp_Obj_local}
\begin{aligned}
    \min_{ z=(x,y_1,\tilde x_1,s_1) } \quad & \varphi(z;\mu):=f_0(x) +  f_1(y_1;\tilde x_1)-\mu\sum_j\ln(s_{1j})\\
        {\rm s.t.} \quad &  c_0(x) \leq 0,\,  \tilde c_1(z):=\begin{pmatrix}
        c_1(y_1;\tilde x_1)+s_1\\
        \tilde x_1-x
    \end{pmatrix}=0.
\end{aligned}
\end{equation}
Let $w:=(u,v_1)$ be primal-dual variables, $z:=(x,y_1,\tilde x_1,s_1)$ be primal variables, and $\LLL(w)=f_0(x)+f_1(y_1;\tilde x_1)+c_0(x)^T\lambda_0+\tilde c_1(z)^T(\lambda_1^T,\eta_1^T)^T$ be the Lagrangian of ~\eqref{eq:undecomp_Obj} after introducing $\tilde x_1$.
Then, the optimality conditions of~\eqref{eq:undecomp_Obj_local} are given by
\begin{equation}\label{eq:undecomp_optcond_local}
    F^C(w;\mu):= 
    \begin{pmatrix}
        \nabla_{x}\LLL(w) \\
        c_0(x)+s_0\\
        \max\{\min\{s_0,\lambda_0\},-s_0,-\lambda_0\} \\
        F_1(w;\mu)
    \end{pmatrix}=0,
\end{equation}
where $s_0$ are slack variables and $\max\{\min\{s_0,\lambda_0\},-s_0,-\lambda_0\}$ captures the complementarity conditions of $c_0$.
Note that $\mu$ enters $F^C$ only by $s_1\circ \lambda_1-\mu e$ in $F_1$, defined in~\eqref{eq:smooth_subprob_optcond}.

Our method achieves superliner convergence by solving a QP subproblem whenever $\mu$ is updated:
when the algorithm updates $\mu^l$ to $\mu^{l+1}$ at the point $z^l$ (Step~\ref{line:decay_mu} in Algorithm~\ref{alg:decomposition_algo_general}), the primal update $\Delta z^l$ is computed by solving 
\begin{subequations}\label{eq:sqp_local}
\begin{align}
    \min_{\Delta z} \quad & \nabla_z \varphi(z^l;\mu^{l+1})^T\Delta z+ \frac{1}{2}\Delta z^T \nabla_{zz}^2\LLL(w^l)\Delta z + \frac{1}{2}\Delta s_1^T(S_1^l)^{-1}\Lambda_1^l\Delta s_1\\
        {\rm s.t.} \quad &  \nabla_x c_0(x^l)^T\Delta x+c_0(x^l) \leq 0, \qquad\qquad\qquad\qquad[\lam_0^+]\label{sqp_local_ineq} \\
        \quad & \nabla_z\tilde c_1(z^l)^T\Delta z+\tilde c_1(z^l)=0, \qquad\qquad\qquad\qquad\, [\lam_1^+,\eta_1^+]
\end{align}
\end{subequations}
where $S_1^l=\text{diag}(s_1^l)$ and $\Lambda_1^l=\text{diag}(\lambda_1^l)$.

\begin{sloppypar}
With an optimal solution $\Delta z$ of \eqref{eq:sqp_local} and multipliers $\lam_0^+$ and $\lam_1^+$, we get the primal-dual step $\Delta w=(\Delta z,\Delta\lambda_0,\Delta\lambda_1,\Delta\eta_1)=(\Delta z,\lam_0^+-\lambda_0,\lam_1^+-\lambda_1,\eta_1^+-\eta_1)$, which we refer to as the \textit{extrapolation step}; see the next subsection for more details.
The new iterate is then computed as $w^{l+1} = w^l + \alpha^l\Delta w^l$, with the step size defined by the fraction-to-the-boundary rule 
\begin{equation}\label{eq:ftbr}
    \alpha^l=\max\left\{\alpha\in(0,1]:
    s_1^l+\alpha\Delta s_1^l\geq(1-\tau^l)s_1^l,\,\,
    \lam_1^l+\alpha\Delta \lam_1^l\geq(1-\tau^l)\lam_1^l
    \right\}
\end{equation}
with a fraction-to-the-boundary parameter $\tau^l\in(0,1)$.
\end{sloppypar}

\begin{algorithm}[t]
\caption{Two-stage decomposition algorithm with extrapolation steps} \label{alg:decomposition_algo_local}
\begin{algorithmic}[1]
\Require{Initial iterate $\tilde u^0$, constants $\mu^0>0$, $c_{\mu,1}\in(0,1)$, $c_{\mu,2}\in(1,2)$, $c_0>0$, $\tau^{\max}>0$.
} 

\State Set $l\gets 0$.
\State Starting from $\tilde u^l$, call the SQP method to solve~\eqref{eq:smooth_master} with $\mu=\mu^l$ to find $u^{l}$ so that $\theta_0(u^l;\mu^l)\leq c_0\mu^l$.
\label{line:master_solver2}
\State Set $w^l=(u^l,v_1^l)$ where $v_1^l$ is the last subproblem solution corresponding to $u^l$.
\label{line:define_z}
\While{$\theta(w^l;\mu^l)\leq c_0\mu^l$}\label{line:while}
\State Set $\mu^{l+1}= \min\{c_{\mu,1}\mu^{l},(\mu^{l})^{c_{\mu,2}}\}$ and $\tau^l=\min\{\tau^{\max},\mu^{l+1}\}$. \label{line:decay_mu_local}
\State Solve \eqref{eq:sqp_local} to get $\Delta w^l$, calculate $\alpha^l$ from \eqref{eq:ftbr}, and set $w^{l+1}=w^l+\alpha^l\Delta w^l$.
\State Set $l\gets l+1$.
\EndWhile
\State Writing $\tilde w^l=(\tilde u^l, \tilde v_1^l)$, extract $\tilde u^l$ as new starting point and go to Step~\ref{line:master_solver2}.
\label{line:initialization2}
\end{algorithmic}
\end{algorithm}

We summarize the implementation of extrapolation steps in Algorithm~\ref{alg:decomposition_algo_local}, as a variant of Algorithm~\ref{alg:decomposition_algo_general}. 
Here, $\theta(w;\mu)=\|F^C(w;\mu)\|$ and $\theta_0(u;\mu)=\|F_0(u;\mu)\|$ with $\na\hat{f}_1(x)=-\eta_1$.
%
Essentially, Steps~\ref{line:define_z}--\ref{line:initialization2} in Algorithm~\ref{alg:decomposition_algo_local} spell out details for Steps~\ref{line:decay_mu}--\ref{line:initialization} in Algorithm~\ref{alg:decomposition_algo_general}.
Because $v_1^l$ in Step~\ref{line:define_z} is a solution of \eqref{eq:obj_smoothing_proj}, we have $\|F_1(z^l;\mu^l)\|=0$ and therefore $\theta_0(u^l;\mu^l)=\theta(z^l;\mu^l)$ at the end of Step~\ref{line:define_z}.  
Consequently, the while-loop is entered at least once.
\SOUT{guaranteeing that $\mu$ is decreased and a new starting point $\tilde u^l$ is computed.}

\subsection{Fast local convergence}
We establish in this subsection the superlinear local convergence of {Algorithm}~\ref{alg:decomposition_algo_local}.
First, we prove a sensitivity result stating that, when $z^l$ is sufficiently close to $z^*$ and $\mu^l$ is sufficiently small, the set of constraints active in \eqref{eq:sqp_local} is identical to the set of constraints active in \eqref{eq:undecomp_Obj_local} at the optimal solutions.
We remark that the classical {sensitivity} result
from Robinson \cite{robinson1974perturbed} is not applicable here, because the Hessian of~\eqref{eq:sqp_local} diverges as $\mu^l\to0$.

Considering \eqref{eq:undecomp_Obj_local}, we let $\AAA^*$ and $\III^*$ denote the active and inactive index sets of $c_0$ at $x^*$ (with $x^*$ from Assumption~\ref{assump:local}),
$c_0^{\AAA^*}$ and $\lambda_0^{\AAA^*}$ denote the active constraints and corresponding multipliers at $x^*$, and $c_0^{\III^*}$ and $\lambda_0^{\III^*}$ are the inactive ones.

\begin{lemma}\label{lem:active_set}
    Suppose Assumption~\ref{assump:full_recourse} and~\ref{assump:local} hold, $z^l$ is sufficiently close to $z^*$ and $\mu^l$ is sufficiently close to 0.
    Then there exists a KKT point $\Delta z^l$ of~\eqref{eq:sqp_local} such that its active set is $\AAA^*$.
\end{lemma}
\begin{proof}
    \begin{sloppypar}
    We first note that the primal-dual solution $w^*$ of \eqref{eq:undecomp_Obj_local} satisfies $F^C(w^*;0)=0$.
    By strict complementarity in Assumption~\ref{assump:local}, there is a neighborhood $B(w^*,r_9)$ such that for $w\in {\bar B(w^*,r_9)}$  
    \end{sloppypar}
    \begin{equation}\label{eq:active_set_for_complementarity}
    \lambda_0^{\AAA^*}>0,\,\,c_0^{\III^*}(x)<0.
    \end{equation}
    We then define a modified version of $F^C$ with the fixed active set $\AAA^*$
    \begin{equation}
        F^{\AAA^*}(w;\mu):=\begin{pmatrix}
        \nabla_{x} \LLL(x)  \\
        -c_0^{\AAA^*}(x) \\
        \lambda_0^{\III^*} \\
        F_1(w;\mu)
    \end{pmatrix}.
    \end{equation}
    Note that $F^{\AAA^*}(w^*;0)=F^C(w^*;0)=0$.

    Since $\lambda_0^{\AAA^*}>0$ for all $w$ in the compact ball {$\bar B(w^*,r_9)$}, there exists a constant $C_3>0$ such that for any update $\|\Delta \lambda_0\|\leq C_3$, $\lambda_0^{\AAA^*}+\Delta\lambda_0^{\AAA^*}>0$.
    Similarly since $c_0^{\III^*}(x)<0$ for all $w\in {\bar B(w^*,r_9)}$ and $c_0$ is $C^2$, by Taylor's theorem there exists $C_4>0$ such that for any update $\|\Delta x\|\leq C_4$, $\nabla c_0^{\III^*}(x)^T\Delta x+ c_0^{\III^*}(x)<0$. 

    Next, we analyze the Newton step of solving $F^{\AAA^*}(w;\mu)$, defined by 
    \begin{equation}\label{eq:newton_step_complementary}
        \Delta \hat{w}:=-\nabla F^{\AAA^*}(w;\mu)^{-1}F^{\AAA^*}(w;\mu).
    \end{equation}
    Since $\mu$ enters $F^{\AAA^*}$ as $s_1\circ \lambda_1-\mu e$, $\nabla F^{\AAA^*}(w;\mu)$ is independent of $\mu$; see \cite{gould2001superlinear}.
    By Assumption~\ref{assump:local}, $\nabla F^{\AAA^*}(w^*;0)$ is nonsingular.
    Since $F^{\AAA^*}(w^*;\mu)=0$ and $F^{\AAA^*}$ is continuous with respect to $\mu$, there exists $\bar \mu>0$, $C_5>0$ and $r_{10}\leq r_9$, such that for $w\in B(w^*,r_{10})$ and $\mu\leq \bar \mu$, $\|\nabla F^{\AAA^*}(w;\mu)^{-1}\|\leq C_5$ and $\|F^{\AAA^*}(w;\mu)\|\leq \frac{1}{C_5}\min\{C_3,C_4\}$.
    Therefore, for $w\in B(w^*,r_{10})$ and $\mu\leq \bar \mu$,
    \begin{equation}\label{ineq:newton_step_bound}
        \|\Delta \hat{w}\|\leq \|\nabla F^{\AAA^*}(w;\mu)^{-1}\|\|F^{\AAA^*}(w;\mu)\|\leq \min\{C_3,C_4\}.
    \end{equation}
    As a result, $w+\Delta \hat{w}$ satisfies~\eqref{eq:active_set_for_complementarity}.

    Finally, we prove the following claim: for $w^l\in B(w^*,r_{10})$ and $\mu^l\leq \bar\mu$, $\Delta \hat{z}^l$ corresponding to $\Delta \hat{w}=(\Delta \hat{z},\Delta\hat{\lambda}_0,\Delta\hat{\lambda}_1,\Delta\hat{\eta}_1)$ discussed above computed at $(w^l;\mu^{l+1})$ is a local solution of~\eqref{eq:sqp_local} with the active set $\AAA^*$.

    To prove the claim, it suffices to show $(\Delta \hat{z},\lambda_0^l+\Delta \hat{\lambda}_0^l,\lambda_1^l+\Delta \hat{\lambda}_1^l,\eta_1^l+\Delta \hat{\eta}_1^l)$ is a KKT point of~\eqref{eq:sqp_local}.
    
    Expanding the Newton system~\eqref{eq:newton_step_complementary}, we have the equivalent set of equations:
    \begin{subequations}
    \begin{align}
        \nabla_{xx}^2\LLL\Delta \hat{x}^l+\nabla_x c_0\Delta \hat{\lambda}_0^l-\Delta \hat{\eta}_1^l+\nabla_x \LLL &=0 \label{localnewton1}\\
        \nabla_x c_0^{\AAA^*}(x^l)^T\Delta \hat{x}^l+c_0^{\AAA^*}(x^l) &=0 \label{localnewton2}\\
        (\Delta{\hat{\lambda}_0^{\III^*}})^l+(\lambda_0^{\III^*})^l &=0 \label{localnewton3}\\
         \nabla_{y_1y_1}^2\LLL\Delta\hat{y}_1^l+\nabla_{y_1\tilde{x}_1}^2\LLL\Delta\hat{\tilde{x}}_1^l+\nabla_{y_1}c_1\Delta\hat{\lambda}_1^l+\nabla_{y_1}\LLL &=0 \label{localnewton4}\\
         \nabla_{\tilde{x}_1\tilde{x}_1}^2\LLL\Delta\hat{\tilde{x}}_1^l+\nabla_{\tilde{x}_1y_1}^2\LLL\Delta\hat{y}_1^l+\nabla_{\tilde{x}_1}c_1\Delta\hat{\lambda}_1^l+\Delta\hat{\eta}_1^l+\nabla_{\tilde{x}_1}\LLL &=0 \label{localnewton5}\\
         \Lambda_1^l \Delta\hat{s}_1+S_1^l\Delta\hat{\lambda}_1^l +s_1^l\circ \lambda_1^l-\mu^{l+1} e &=0 \label{localnewton6}\\
         \nabla_{y_1}c_1^T \Delta\hat{y}_1^l+\nabla_{\tilde{x}_1}c_1^T\Delta\hat{\tilde{x}}_1^l+\Delta\hat{s}_1^l + c_1(y^l;x^l) + s_1^l & = 0 \label{localnewton7}\\
         -\Delta\hat{x}^l+\Delta\hat{\tilde{x}}_1^l+\tilde x_1^l-x^l &=0. \label{localnewton8}
    \end{align}
\end{subequations}
    \eqref{localnewton1}, \eqref{localnewton4}, and \eqref{localnewton5} are the stationarity in the KKT conditions of~\eqref{eq:sqp_local} with respect to $\Delta x$, $\Delta y_1$, and $\Delta \tilde x_1$;
    \eqref{localnewton2} indicates that for indices in $\AAA^*$, $\nabla_x c_0(x^l)^T\Delta x+c_0(x^l) \leq 0$ is active;
    \eqref{localnewton3} implies $(\lambda_0^{\III^*})^{l+1}=0$;
    \eqref{localnewton7} and \eqref{localnewton8} give the primal feasibility of the equalities in~\eqref{eq:sqp_local}.
    By~\eqref{ineq:newton_step_bound} and the discussion before~\eqref{eq:newton_step_complementary}, $c_0^{\III^*}(x^{l+1})<0$ and $(\lambda_0^{\AAA^*})^{l+1}>0$.
    It follows that primal-dual feasibility and complementary slackness of the inequalities are both satisfied in the KKT conditions of~\eqref{eq:sqp_local}.
    Furthermore, the active set of $\Delta \hat{z}^l$ is exactly $\AAA^*$.

    It remains to check the stationarity with respect to $\Delta s_1$.
    By~\eqref{localnewton6}, we have
    \[
        \mu^{l+1}e = \Lambda_1^l \Delta\hat{s}_1+S_1^l(\Delta\hat{\lambda}_1^l+ \lambda_1^l)  
       \Leftrightarrow  -\mu^{l+1}(S_1^l)^{-1}e+(S_1^l)^{-1}\Lambda_1^l \Delta\hat{s}_1+(\Delta\hat{\lambda}_1^l+ \lambda_1^l)=0,
    \]
    which is the stationarity of $s_1$ for~\eqref{eq:sqp_local} with multipliers as $\Delta\hat{\lambda}_1^l+ \lambda_1^l$.
    Therefore the KKT conditions of~\eqref{eq:sqp_local} are all verified and this finishes the proof.
\end{proof}

As a result of Lemma~\ref{lem:active_set}, for a local analysis near $w^*$ we are able to replace $c_0(x)\leq0$ in~\eqref{eq:undecomp_Obj_local} by $c^{\AAA^*}_0(x)=0$ without changing the steps that the algorithm takes.
After defining
$F_0^{\AAA^*}(w;\mu)=\begin{pmatrix}
    \nabla f_0(x)-\eta_1+\na c^{\AAA^*}_0(x)\lambda_0 \\
     c^{\AAA^*}_0(x)
\end{pmatrix}$ we can simplify~\eqref{eq:undecomp_optcond_local} as
\begin{equation}\label{eq:undecomp_optcond_primaldual}
F(w;\mu) = 
    \begin{pmatrix}
        F_0^{\AAA^*}(w;\mu) \\ F_1(w;\mu)
    \end{pmatrix}=0.
\end{equation}
It then follows that solving the extrapolation step $\Delta w^l$ as a solution of the equality-constrained variant of \eqref{eq:sqp_local} is equivalent to compute a Newton step of solving~\eqref{eq:undecomp_optcond_primaldual}:
\begin{equation}\label{eq:bigNewton}
\na F(w^l)^T\Delta w^l = - F(w^l,\mu^{l+1}).
\end{equation}
Here, the argument $\mu$ in $\na F(w^l)$ is intentionally dropped, since $\mu$ appears only as a constant in $F$.

$\Delta w^l$ is called an \textit{extrapolation step} following \cite{gould2001superlinear}, since it was shown in \cite{gould2001superlinear} that $\Delta w^l$ can be interpreted as the composition of a full Newton step at $w^l$ for $\mu^l$ and an extrapolation along the central path from $\mu^l$ to $\mu^{l+1}$.
Consequently, one can prove the superlinear local convergence by the analysis from basic primal-dual interior-point methods, e.g., \cite{byrd1997local}.
Finally, let us state the local convergence result.
\begin{theorem}[Superlinear local convergence]
\label{thm:superlinear_local}
Suppose Assumption~\ref{assump:full_recourse} and \ref{assump:local} hold and that Algorithm~\ref{alg:decomposition_algo_local}\SOUT{, with the computation of the extrapolation step described in Section~\ref{sec:extrapolation_implementation},} encounters an iterate $w^l$ sufficiently close to $w^*$ in Step~\ref{line:define_z} for a sufficiently small value of $\mu^l$.  Then the algorithm will remain in the while-loop and $w^l$ converges to $w^*$ at a superlinear rate.
\end{theorem}

\begin{proof}
By Lemma~\ref{lem:active_set}, the QP subproblem solver implicitly identifies the active constraints $c^{\AAA^*}_0$ at $z^*$.  
As a consequence, the steps calculated by the algorithm satisfy \eqref{eq:bigNewton} and the while loop executes the basic interior-point algorithm analyzed in \cite{byrd1997local}.  Assumption~\ref{assump:local} implies the assumptions necessary for the analysis in \cite{byrd1997local}, and the discussions in Section 4 and 5 in \cite{byrd1997local} imply the claim of this theorem.
\end{proof}

\subsection{Extrapolation step within the decomposition framework}
\label{sec:extrapolation_implementation}
In this section, we discuss how to efficiently compute $\Delta w^l$.
Solving~\eqref{eq:sqp_local} involves the master and all subproblem variables and can become extremely large.
Our decomposition technique allows the parallel solution of all subproblems efficiently. We show that a Schur complement approach makes it possible to reuse the computations 
in an implementation of Algorithm~\ref{alg:decomposition_algo_general}. 
As a result, very little programming effort\Footnote{\ew{computing effort (programming effort could be about coding, but here we are talking about the actual computation)}} is required to integrate an extrapolation step.


Let $\AAA$ be the active index set of~\eqref{sqp_local_ineq} at the current iterate $w$, and write $F_0(w;\mu)=\begin{pmatrix}
    \nabla f_0(x)-\eta_1+\na c^{\AAA}_0(x)\lambda_0 \\
     c^{\AAA}_0(x)
\end{pmatrix}$.
Therefore,\SOUT{as established in the proof of Lemma~\ref{lem:active_set},} the solution of~\eqref{eq:sqp_local} is equivalent to the Newton step in~\eqref{eq:bigNewton} with $\AAA$.
Omitting arguments and iteration counters, using the Schur complement, the solution $\Delta w=(\Delta u, \Delta v_1)$ can be calculated by:
\begin{subequations}
    \begin{align}
        \na_{v_1}F_1^T \cdot \Delta v_1^{(1)} & = - F_1 \label{eq:schur1}\\
        \left( \na_uF_0 - \na_{v_1}F_0(\na_{v_1}F_1)^{-1}\na_{u}F_1^T \right)^T \cdot \Delta u & = - F_0 - \na_{v_1}F_0\cdot \Delta v_1^{(1)} \label{eq:schur2}\\
        \na_{v_1}F_1^T \cdot\Delta v_1^{(2)} & = - \na_u F_1 \cdot \Delta u \label{eq:schur3}\\
        \Delta v_1 & =\Delta v_1^{(1)} + \Delta v_1^{(2)}. \label{eq:schur4}
    \end{align}
\end{subequations}

As before, we assume that $\na_{v_1}F_1$ is nonsingular (see Section~\ref{sec:nondegenerate_justification}).
By the definition of $F_1$, \eqref{eq:schur1} is identical to the linear system (Step \ref{line:Newton_step} in Algorithm~\ref{alg:newton}) that is solved internally in the interior-point subproblem solver, and can therefore be computed without additional programming efforts.

Considering \eqref{eq:schur2}, we note that
\[
\na_{v_1}F_0 = \na_uF_1^T = 
\begin{bmatrix}
    0 & 0 & 0 & 0 & -I\\
    0 & 0 & 0 & 0 & 0
\end{bmatrix}.
\]
Writing out $\na_{v_1}F_0(\na_{v_1}F_1)^{-1}\na_{u}F_1^T$ in detail shows that this matrix is zero except for one block, which we denote as $-\tilde H_1$.  The computational procedure to obtain $-\tilde H_1$ is identical to computing $\na^2\hat f_1(x,\mu)$ via \eqref{eq:na_vstar}, again without additional programming efforts.
Letting $\tilde g_1=-(\eta_1 + \Delta\tilde\eta_1)$, \eqref{eq:schur2} becomes
\[
\begin{bmatrix}
    \na_{xx}^2 \LLL(x,\lam_0) + \tilde H_1 & \na c^\AAA_0(x) \\
    \na c^\AAA_0(x)^T & 0
\end{bmatrix}
\begin{pmatrix}
    \Delta x \\ \Delta \lam_0    
\end{pmatrix}
=-
\begin{pmatrix}
    \na f_0(x)+\tilde g_1+ \na c^\AAA_0(x)\lam_0\\
     c^\AAA_0(x)
\end{pmatrix}.
\]
Since $\mathcal A$ is assumed to be the active set for~\eqref{sqp_local_ineq}, this is equivalent to computing a stationary point for
\begin{subequations}\label{eq:QP_extrapolation}
\begin{align}
    \min_{\Delta x}\quad & (\na f_0(x) +\tilde g_1)^T\Delta x + \tfrac12 \Delta x^T (\na_{xx}^2 \LLL(x,\lam_0) + \tilde H_1)\Delta x\\
    \text{{\rm s.t.}}\quad & \na c_0(x)^T\Delta x + c_0(x) \leq 0,\label{eq:QP_extrapolation2}
\end{align}
    \end{subequations}
which is the SQP step computation (see~\eqref{eq:sqp_qp}), except that it uses a ``fake'' subproblem gradient $\tilde g_1$ and Hessian $\tilde H_1$ of the subproblem.
{Crucially, although $\AAA$ is introduced for derivation of~\eqref{eq:QP_extrapolation}, we do not require $\AAA$ for computing $\tilde g_1$ and $\tilde H_1$.}
%
%
Since this QP subproblem is already part of the SQP solver, no additional programming work is required.
If the SQP method uses a trust region, it is well known that the trust region is inactive close to the optimal solution under Assumption~\ref{assump:local} and does not affect the solution \cite{conn2000trust}.

Solving \eqref{eq:schur3} can again be accomplished with the internal linear algebra in the interior-point solver, with $\Delta u$ sent from the master solver.
Finally, the overall subproblem step in \eqref{eq:schur4} is sent to the master problem solver so that it is able to compute $\theta(w^l;\mu^l)$ for Step~\ref{line:while} of Algorithm~\ref{alg:decomposition_algo_local}.
\section{Numerical experiments}
\label{sec:numerical}

\subsection{Implementation}
We utilized the C++ implementation of the decomposition algorithm developed as part of Luo's thesis \cite{luo2023efficient}.
The outer loop of Algorithm~\ref{alg:decomposition_algo_general} was run with $\mu_0=0.1$, $c_0=0.1$, $c_{\mu,1}=0.2$ and $c_{\mu,2}=1.5$.  The algorithm terminates when the smoothing parameter reaches $\mu_{tol}=10^{-6}$, where Step~\ref{line:decay_mu} is implemented as $\mu^{l}\gets \max\{\min\{c_{\mu,1}\mu^{l},(\mu^{l})^{c_{\mu,2}}\},\mu_{tol}\}$.

The master problem solver is an advanced version of the S$\ell_1$QP method analyzed in Section~\ref{sec:SQP_global_convergence_fixed_mu} that includes means to update the penalty parameter $\pi$ in \eqref{eq:phi}; for details see \cite[Chapter 3]{luo2023efficient}. The QP subproblems are solved with the primal-dual interior-point method {\tt Ipopt} \cite{wachter2006implementation}.
{\tt Ipopt} is also used to solve the subproblems.  
Due to the object-oriented design of {\tt Ipopt}, one can easily access the internal linear algebra routines in {\tt Ipopt} and use them to efficiently solve \eqref{eq:na_vstar} and similar systems.

When the smoothing parameter is decreased in Step~\ref{line:decay_mu}, we begin the extrapolation procedure detailed in Algorithm~\ref{alg:decomposition_algo_local} in Section~\ref{sec:local_convergence}.

Before starting Algorithm~\ref{alg:decomposition_algo_general}, our implementation solves the first-stage problem once\SOUT{with {\tt Ipopt}}, where the shared variables with the subproblems are fixed.\SOUT{ to the user-given initial values.}
This presolve provides a good primal-dual starting point for Algorithm~\ref{alg:decomposition_algo_general}.

The C++ implementation in \cite{luo2023efficient} includes several interfaces, including a convenient AMPL interface that allows one to pose the master problem and the subproblems as separate AMPL models \cite{fourer1990ampl}.  In addition, the subproblems can be solved in parallel with multiple threads using OpenMP.  The experiments reported here were executed on a Linux desktop with 32GB of RAM and 2.9GHz 8-core Intel Core i7-10700 CPU.

{
In addition to the results below, we also ran our algorithm on those small-scale examples in Section~\ref{sec:smoothing} and~\ref{sec:behaviors_nonconvex}; see Appendix~\ref{appendix:supp_experiments}.
}

\subsection{Two-stage quadratically constrained quadratic programs} 
\label{sec:numerical_qcqp}
We explore the performance of the decomposition algorithm using large-scale instances with nonconvex subproblems.  We randomly generated Quadratically Constrained Quadratic Programs (QCQP) instances in the following form:
\begin{equation}
\begin{aligned}
   \min_{x} \quad & \half x^T Q_0 x + c_0^T x + \sum_{i=1}^N \hat f_i(x)\\ 
           \text{s.t.}\quad &  \half x^T Q_{0j} x + c_{0j}^T x +r_{0j} \leq 0, \quad \quad j = 1, \ldots, m_0, \end{aligned}
\end{equation}
with  
\begin{equation}
\begin{aligned}
    \hat f_i(x) =\min_{y_i,\tilde x_i, p_i,t_i} \quad  & \half y_i^T Q_i y_i + c_i^T y_i + \rho \sum_{{j=1}}^{n_c}(p_{ij} {+} t_{ij})\\
    \text{s.t.}\quad &  \half y_i^T Q_{ij} y_i + c_{ij}^T y_i + b_{ij}^T \tilde x_i +r_{ij} \leq 0, \quad j = 1, \ldots, m_i \\
    & Px-\tilde x_i = p_i - t_i ,\quad -50\leq y_i\leq 50, \quad p_i,t_i\geq0.
\end{aligned}
\end{equation}
Here, $r_{ij}\in\RR$, $c_{ij}\in\RR^{n_i}$, $b_{ij}\in\RR^{n_c}$, where $n_c$ is the number of first-stage variables appearing in the second-stage. The projection matrix $P=[I_{n_c}\, 0_{n_0-n_c}]$ extracts $n_c$ first-stage variables corresponding to the copy $\tilde x_i$.  
The slack variables $p_i$ and $t_i$ are penalized in the objective with the weight $\rho$,  so that the subproblems are always feasible for any $x$.
In our experiments we set $\rho=100$, which is large enough to ensure that $Px=\tilde x_i$ at the optimal solution.
The matrix $Q_0$ is a diagonal matrix with entries between 0.1 and 1, and $Q_i$ ($i=1,\ldots N$) are diagonal with entries between -1 and 1. 
See the full procedure for generating the test data in \cite[Algorithm~10]{luo2023efficient}.\SOUT{with the difference that the diagonal entries of $Q_i$ are chosen between 0.1 and 1 in that procedure. 
with the density parameter $d$ set as 0.05.}

\begin{figure}[t]
    \centering
    \includegraphics[width = 0.75\textwidth]{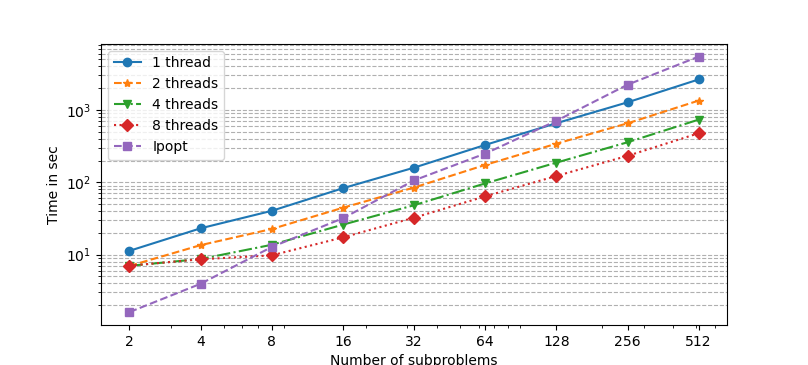}
    \caption{Wallclock computation times for QCQP instances.}
    \label{fig:QCQP_times}
\end{figure}

We ran our experiments with $n_i=250$ and $m_i=500$ for each $i$ and $n_c=10$.
The monolithic formulation \eqref{eq:undecomp_Obj} of the largest instances has 128,000 variables and 256,500 constraints, with a total of 5,638,500 nonzeros in the constraint Jacobian and 128,250 nonzeros in the Lagrangian Hessian.
Figure~\ref{fig:QCQP_times} shows computation time as a function of the number of subproblems, averaged over 10 runs with random data.
With less than 0.1s, the time for the initial master problem presolve is negligible.

The decomposition algorithm was run in parallel with 1, 2, 4, and 8 threads.  It can be observed that the computation time increases linearly with the number of subproblems.  
\SOUT{The parallel speedup is also very favorable.}On average, for the largest instances with $N=512$, the computation time was reduced by a factor of 2.0, 3.6, and 5.5 for the 2, 4, and 8 threads respectively, compared to the single-thread performance.

To showcase the computational benefit of the decomposition approach, Figure~\ref{fig:QCQP_times} also includes the computation time required by {\tt Ipopt} to solve the undecomposed monolithic instances.  As expected, for small instances, {\tt Ipopt} is much faster, but the computation increases at a rate of about $\mathcal{O}(N^{1.5})$ as the size of the problem grows.
As a consequence, the decomposition method is faster than {\tt Ipopt} when $N>128$, and when 8 threads are available, it is already faster when $N=8$.

The worse-than-linear increase in the time of {\tt Ipopt} partially stems from a rise in {\tt Ipopt} iterations, averaging 27.7 for $N=1$ to 193.5 for $N=512$, due to the nonnegative curvature encountered during the optimization.
In contrast, the iteration counts for the decomposition method remain unaffected by the problem size, with 20-22 SQP iterations and 233-260 {\tt Ipopt} iterations per second-stage problem across all sizes.
%
%
This indicates that an approach that parallelizes linear algebra within an interior-point method applied directly to the monolithic formulation \cite{lubin2011scalable,zavala2008interior} may scale less favorably with $N$ than the proposed framework for nonconvex problems.
Despite the nonconvexity of the problem, the final objective values of {\tt Ipopt} and our algorithm are identical.
%

\section{Conclusions}
\Footnote{AW: Still to do}
In this work, we studied the convergence properties of a framework for nonlinear nonconvex two-stage optimization with nonlinear constraints.
The approach can be extended to instances where \eqref{eq:intro_master} and \eqref{eq:intro_subproblem} include equality constraints, as our analysis remains valid provided the gradients of the constraints in the second-stage are linearly independent.
Our method allows flexibility in extending the algorithm by substituting master and subproblem solvers in Algorithm~\ref{alg:decomposition_algo_general} with off-the-shelf options.  For example, an interior-point solver could be applied to the master problem. 
Furthermore, the smoothing of the second-stage problem, which relies on applying the implicit function theorem, could be applied to other problem structures that give rise to perturbed optimality conditions, such as those with second-order cones of semi-definite matrix constraints.
%
Finally, we remark that several techniques developed here might also be applicable to nonconvex min-max and bilevel optimization, but a detailed exploration of these extensions is left for future research.



\bibliographystyle{siamplain}
\bibliography{references}

\newpage
\appendix
\section{Supplementary experimental details}\label{appendix:supp_experiments}
We verify the performance and theory of our proposed algorithm on the examples introduced in Section~\ref{sec:smoothing} and~\ref{sec:behaviors_nonconvex}.
In some cases, the subproblem solver fails to converge for some trial points $x^k+p^k$ of the master problem algorithm.  
For instance while solving Example~\ref{ex:bilinear}, {\tt Ipopt} converges to infeasible stationary points of the subproblem in a few early iterations. 
In that event, the master S$\ell_1$QP solver rejects such trial points and reduces the trust region radius.  
In this way, the trust region radius eventually becomes sufficiently small, so that the warm start strategy guarantees the subproblem solutions correspond to one consistent solution map; see Proposition~\ref{prop:f_is_C2_single}.
\subsection{Small-scale examples from Section~\ref{sec:behaviors_nonconvex}}
\label{sec:numerical_small_scale}
\paragraph{Example~\ref{ex:sol_smoothing}}

In Section~\ref{sec:smoothing}, we introduced Example~\ref{ex:sol_smoothing} to demonstrate that solution smoothing might result in spurious nonconvexity.  When we ran the decomposition algorithm for this example, we observed that it converges to the true optimal solution $x^*=2$ and is not attracted to the spurious solution $x=0$.  
The reason is that the QP solver within the SQP method always finds the global minimizer $x=2$ of the nonconvex QP. 
%
However, Chapter 4.7.3 in \cite{luo2023efficient} presents an instance with linear subproblems in which the method converges to a spurious solution that is not a local minimizer.  In addition, solution smoothing induces negative curvature in the master problem which leads to  more SQP iterations compared to objective smoothing.

\paragraph{Example~\ref{ex:bilinear}}

If we start Algorithm~\ref{alg:decomposition_algo_general} with $x^0=0.4$ and use $y=0$ as the starting point for {\tt Ipopt} in the subproblem, {\tt Ipopt} converges to $\tilde y^0=-0.278$ and returns the corresponding values for $\hat f_1(x^0)$, $\na\hat f_1(x^0)$, and $\na^2\hat f_1(x^0)$.  The next iterate of the SQP solver is $x^1=2$ and {\tt Ipopt} computes $\tilde y^1=-1.89$.  After one additional iteration, the (relaxed) tolerance for the SQP solver is reached and $\mu$ is decreased.  Next, the extrapolation step $\Delta x$ of the SQP iterate is taken, but a single {\tt Ipopt} iteration does not satisfy the new tolerance.  Instead, the regular SQP algorithm is resumed, and after one iteration $\mu$ is decreased again. This is repeated one more time.  After that, the extrapolation step is accepted for each decrease of $\mu$ and the method converges towards $x^*=2$ and $y^*=-2$.  In all, the subproblem is solved 6 times, requiring a total number of 29 {\tt Ipopt} iterations.  Importantly, for all subproblem calls, {\tt Ipopt} returns optimal solutions corresponding to the red solution map in Figure~\ref{fig:linear_solution_maps_smooth}.

On the other hand, if we start the algorithm using $y=-2$ as the starting point for {\tt Ipopt}, {\tt Ipopt} converges to $\tilde y^0=-2.32$, corresponding to the blue solution map in Figure~\ref{fig:linear_solution_maps_smooth}.  The trial point in the next SQP iteration is $\tilde x^0=x^0+p^0=2$.  This time, when {\tt Ipopt} tries to solve the subproblem with $y=-2.32$ as the starting point, it fails to converge and reports that an infeasible stationary point is found.  Consequently, the SQP solver reduces the trust region and sends $\tilde x^1=1.2$ as the next trial point to {\tt Ipopt}.  {\tt Ipopt} fails again, so the SQP solver reduces the trust region again, with $\tilde x^2=0.8$ sent to {\tt Ipopt}.  This time, {\tt Ipopt} computes $\tilde y^2=-2.74$ as the subproblem solution.  From then on, the SQP solver only sends trial point $\tilde x^k<1$ to {\tt Ipopt}, which converges to solutions corresponding to the blue solution map, and the algorithm converges towards $x^*=1$ and $y^*=-3$.

\paragraph{Example~\ref{ex:curvature}}

{For this instance, Algorithm~\ref{alg:decomposition_algo_general}  converges towards $x^*=-1$ and $y^*=2$ when $x^0<0$ and $y^0\geq0$, and towards $x^*=-1$ and $y^*=-1$ when $x^0\leq-0.5$ and $y^0\geq0.1$, for instance.
When $x^0>0$, Algorithm~\ref{alg:decomposition_algo_general} appears to always converge towards $x^*=-1$ and $y^*=2$, even if $y^0$ is chosen very closely to $-1$.}


\section{Supplementary proofs}\label{appendix:supp_proofs}
The details of the induction argument in the proof of Lemma~\ref{lem:newton_region_of_attraction} are given by:
\begin{proof}
    \begin{sloppypar}
    We show by a strong form of induction that for $j=0,1,2,\cdots$, $\|v_i^{j+1}-v_i^*\|\leq (1-\frac{\beta_{\min}}{2})\|v_i^j-v_i^*\|$.
    Note that since $x,x^+\in B({\bar x},r_4)$, we have $\|v_i^0-v_i^*\|=\|v_i^*(x;(\bar x,\bar v_i))-v_i^*(x^+;(\bar x,\bar v_i))\|\leq C$ {as well as $\|v_i^0-\bar v_i\|=\|v_i^*(x;(\bar x,\bar v_i))-v_i^*(\bar x;(\bar x,\bar v_i))\|\leq C$}.
    This shows that the initial point is close to a local stationary point.
    {Furthermore, one has $\|(x^+,v_i^0)-(\bar x,\bar v_i)\|\leq \|x^+-\bar x\|+\|v_i^0-\bar v_i\|\leq 2r_3+C\leq r_2$ so that~\eqref{ineq:newton_decay} is applicable.}
    Therefore for $j=0$, one has that
    \end{sloppypar}
    \begin{equation}
    \label{eq:strong_induction_j=0}
    \begin{aligned}
        \|v_i^1-v_i^*\| &= \|v_i^0+\beta^0\Delta v_i^0-v_i^*\| \\
        & = \|\beta^0(v_i^0+\Delta v_i^0-v_i^*) +(1-\beta^0)(v_i^0-v_i^*) \| \\
        & \leq \beta^0\|\hat{v}_i^1-v_i^*\| +(1-\beta^0)\| v_i^0-v_i^* \| \\
        & \leq \tfrac{\beta^0}{2C}\|v_i^0-v_i^*\|^2 + (1-\beta^0)\| v_i^0-v_i^* \| \\
        & \leq \tfrac{\beta^0}{2}\| v_i^0-v_i^* \|+ (1-\beta^0)\| v_i^0-v_i^* \| \\
        & = \left( 1-\tfrac{\beta^0}{2} \right)\| v_i^0-v_i^* \| \leq \left( 1-\tfrac{\beta_{\min}}{2} \right)\| v_i^0-v_i^* \| ,
    \end{aligned}
    \end{equation}
    where the first inequality is from triangle inequality; the second follows from~\eqref{ineq:newton_decay}; and the third is from the property of $B(x,r_4)$ as argued in the paragraph above.
    Thus the statement is true for $j=0$.
    
    Next by a strong form of induction hypothesis, let us assume the inequality holds for $j=0,1,\cdots,J-1$.
    It then follows from $1-\frac{\beta_{\min}}{2}<1$ that, for $0\leq j\leq J-1$, $\|v_i^{j+1}-v_i^*\|$ is a decreasing sequence and thus 
    \begin{equation}
    \label{ineq:induction_hypo}
        \|v_i^{j+1}-v_i^*\|\leq C,\quad\text{for all }j\leq J-1.
    \end{equation}
    {It follows that for $j\leq J$, $\|(x^+,v_i^j)-(\bar x,\bar v_i)\|\leq \|x^+-\bar x\|+\|v_i^j-\bar v_i\|\leq 2r_3+\|v_i^j-v_i^*\|+\|v_i^*-\bar v_i\|\leq 2r_3+2C\leq r_2$, so that~\eqref{ineq:newton_decay} is applicable for all $j\leq J$.}
    Then, similar to the derivation of~\eqref{eq:strong_induction_j=0}, by~\eqref{ineq:newton_decay} and~\eqref{ineq:induction_hypo}
     \begin{equation}
    \begin{aligned}
        \|v_i^{J+1}-v_i^*\| & \leq \beta^J\|\hat{v}_i^{J+1}-v_i^*\| +(1-\beta^J)\| v_i^J-v_i^* \| \\
        & \leq \frac{\beta^J}{2C}\|v_i^J-v_i^*\|^2 + (1-\beta^J)\| v_i^J-v_i^* \| \\
        & \leq \left( 1-\tfrac{\beta^J}{2} \right)\| v_i^J-v_i^* \|  \leq \left( 1-\tfrac{\beta_{\min}}{2} \right)\| v_i^J-v_i^* \|,
    \end{aligned}
    \end{equation}
    which proves the inequality for $j=J$.
    Therefore $\|v_i^{j+1}-v_i^*\|\leq (1-\frac{\beta_{\min}}{2})\|v_i^j-v_i^*\|$ for all $j$ by induction.
\end{proof}

\end{document}